\documentclass[11pt]{article}
\usepackage{amssymb}
\usepackage{amsmath}
\usepackage{amsthm}
\usepackage[textwidth=16cm, textheight=24cm, marginratio=1:1]{geometry}

\usepackage[english]{babel}
\usepackage[utf8]{inputenc}
\usepackage[T1]{fontenc}
\usepackage{todonotes}
\usepackage{xcolor}
\usepackage{xfrac}
\usepackage{bbold}
\usepackage{mathtools}
\usepackage{textcomp}
\usepackage{booktabs}

\usepackage[pdfborder={0 0 0}]{hyperref}
\hypersetup{
colorlinks,
linkcolor={red!80!black},
citecolor={blue!80!black},
urlcolor={blue!80!black}
}
\usepackage{setspace}
\usepackage[maxbibnames=99]{biblatex}
\usepackage{csquotes}

\newtheorem{theorem}{Theorem}[section]
\newtheorem*{theorem-non}{Theorem}
\newtheorem{corollary}[theorem]{Corollary}
\newtheorem{lemma}[theorem]{Lemma}

\newtheorem{proposition}[theorem]{Proposition}
\theoremstyle{definition}
\newtheorem{definition}[theorem]{Definition}
\newtheorem{remark}[theorem]{Remark}
\newtheorem{example}[theorem]{Example}

\DeclareMathOperator{\Span}{span}
\DeclareMathOperator{\Gr}{gr}
\DeclareMathOperator{\Rank}{rank}

\title{Iarrobino's symmetric decomposition for self-dual modules}
\author{Maciej Wojtala}
\date{\today{}}

\begin{document}

\maketitle

\begin{abstract}

We generalize Iarrobino’s symmetric decomposition for the associated graded algebra of an Artinian Gorenstein algebra to a symmetric decomposition of finite-length self-dual modules over a local algebra, and we deduce consequences for the Hilbert functions of such self-dual modules.
We classify the local Hilbert functions for small degree modules.
We generalize Kunte's criterion for self-duality in terms of Macaulay's inverse systems.
\end{abstract}

\section{Introduction}
\label{section-Introduction}

\subsection{Main results}
\label{Subsection-main-results}

In our work, we introduce three main results: a generalization of Iarrobino’s
symmetric decomposition from Artinian Gorenstein algebras to finite-length self-dual modules over a local algebra; the generalization of Kunte’s criterion for self-duality in terms of Macaulay's inverse systems; and a classification of possible local Hilbert functions of self-dual modules of degree (length) at most eight.

A. Iarrobino introduced Iarrobino's symmetric decomposition for self-dual algebras (known also as Artinian Gorenstein algebras) \cite{Iarrobino1994AssociatedGA, IARROBINO2021106496}.
It is used in the deformation theory \cite{casnati2015irreducibility}, the theory of symmetric matrices \cite{Kosir}, the classical algebraic geometry \cite{schreyer2001GeometryAA, ArzZai22}, and the theory of higher-dimensional Gorenstein algebras \cite{ELIAS2017306, Jelisiejew2015ClassifyingLA}.
This \cite{Iarrobino1994AssociatedGA} work of A. Iarrobino is relevant for algebraic combinatorics, Lefschetz properties, theory of secant varieties and sums of powers, etc. It is widely applied and cited over $110$ times, including such classical books and papers as \cite{stanley2007combinatorics, harima2013lefschetz, harima2003weak, ranestad2000varieties, maeno2009lefschetz, bernardi2013cactus, huneke1999hyman, elias2012isomorphism, wiebe2004lefschetz, schreyer2001GeometryAA, kustin2017equations, casnati2015irreducibility}.
It keeps being relevant today; it is employed for example in \cite{jelisiejew2023open, mcdaniel2021free, ssevviiri2024applications, failla2021weak, szafarczyk2023new, bernardi2018polynomials, buczynski2019constructions, jelisiejew2018vsps, alper2018associated, elias2018inverse}.
J. Jelisiejew has recently developed a deeper geometric theory
corresponding to the symmetric decomposition of local Gorenstein
algebras \cite{jelisiejew2025iarrobinoschemeselfdualanalogue}.

The main claim in the result is that the associated graded algebra of a self-dual algebra (Artinian Gorenstein algebra) has a decreasing sequence of ideals whose successive quotients are reflexive modules. The consequence is that the local Hilbert function of the algebra can be decomposed into the sum of symmetric sequences.
This result introduces criteria that the local Hilbert function of a self-dual algebra must satisfy.

F. Macaulay introduced the bound on the Hilbert functions of algebras \cite{Macaulay}.
This result was generalized to modules by H. Hulett \cite{Hulett}.
Iarrobino's symmetric decomposition in combination with Macaulay's Bound is a very useful tool in showing that a sequence is not the local Hilbert function of any self-dual algebra.
S. Kleiman and J.O. Kleppe generalized the Macaulay duality over more general rings \cite{kleiman2025macaulay}. We leave the generalization of our study in that direction for future work.

We generalize the result of A. Iarrobino to the case of self-dual modules.
To the best of our knowledge, this is the first time that Iarrobino's symmetric decomposition is considered in terms of modules.

Throughout this work, let $\mathbb{K}$ denote a field of characteristic 0. It would be interesting to understand the case of prime characteristic, but we leave this to future work.
Let $R$ denote a finite-dimensional local $\mathbb{K}$-algebra with a residue field $\mathbb{K}$, and let $M$ denote an $R$-module that is a finite-dimensional $\mathbb{K}$-vector space.

On a dual vector space $M^*$ we introduce an $R$-module structure by defining
\begin{equation*}
    r \cdot \phi(m) = \phi(rm).
\end{equation*}

\begin{definition}
\label{definition-self-dual}
    A module $M$ is self-dual if it is isomorphic with $M^*$ as an $R$-module.
\end{definition}

A case of interest is when $R$ is a local ring and $M$ is a self-dual module, although some results are shown without those assumptions.

\begin{definition}
\label{definition-negative-powers}
    For an ideal $\mathfrak{a}$ in $R$ we denote $\mathfrak{a}^z = (1)$ for any negative integer $z$.
\end{definition}

\begin{remark}
\label{lemma-finite-Hilbert}
    Let $R$ be local and let $\mathfrak{m}$ be its (only) maximal ideal.
    There exists the largest natural number $d$ such that $\mathfrak{m}^{d}M \neq 0$, called the socle degree of $M$.
\end{remark}

Put informally, we show for a module $M$ of the socle degree $d$, that the local Hilbert function $h_M$ is a sum
\begin{equation}
\label{equation-hilbert-decomposition}
   h_M = \sum_{i=0}^d \Delta_i,
\end{equation}
where $\Delta_i$ is symmetric around $\frac{d-i}{2}$.

We now introduce the necessary definitions (see also Section~\ref{section-Iarrobino-s-decomposition}).
Let $R$ be a local ring and let $\mathfrak{m}$ be its maximal ideal.
Let $d$ be the socle degree of $M$.
By $\Gr R$ we denote the associated graded ring of $R$, i.e.
\[ \Gr R = \bigoplus_{k=0}^d \frac{\mathfrak{m}^k}{\mathfrak{m}^{k+1}}. \]
By $\Gr M$ we denote the associated graded module of the module $M$, i.e.
\[ \Gr M = \bigoplus_{k=0}^d \frac{\mathfrak{m}^k M}{\mathfrak{m}^{k+1} M}, \]
which is a ($\Gr R$)-module in a natural way.
For an ideal $\mathfrak{a}$ in $R$ we denote by $(0 : \mathfrak{a})_M$ the annihilator of $\mathfrak{a}$ in $M$, i.e. the set of all elements $x \in M$, such that $\mathfrak{a} x = 0$. The annihilator $(0 : \mathfrak{a})_M$ is a submodule of $M$.

Let us denote 
\[
A_{k, \: l} = \mathfrak{m}^{k}M \cap (0 : \mathfrak{m}^{l})_M \text{ for all integers } k, \: l \text{ (we follow the convention from Definition \ref{definition-negative-powers})},
\]
\begin{equation}
\label{equation-q-def}
    Q_{k, \: l} = \frac{A_{k, \: l}}{A_{k + 1, \: l} + A_{k, \: l - 1}} \text{ for all } 0 \leq k, \: l \leq d + 1.
\end{equation}
Let us also define
\[
\Delta_s(t) = 
   \begin{cases*}
      \Rank_{\mathbb{K}} \: Q_{t, \: d+1-(s+t)} & for $0 \leq s \leq d, \: 0 \leq t \leq d-s$ \\
      0           & otherwise.
    \end{cases*}
\]

\begin{remark}
\label{remark-Q-delta}
    We have $\Rank_{\mathbb{K}} \: Q_{k, \: l} = \Delta_{d+1 - (k+l)}(k)$.
\end{remark}

For $0 \leq s \leq d$, let us define the vector space
\[D_s = \bigoplus_{0 \leq t \leq d-s} Q_{t, \: d+1 - (s+t)} = \bigoplus_{k+l = d+1-s} Q_{k, \: l}.\]

For $0 \leq s \leq d$ let $C_s^i \subseteq \sfrac{\mathfrak{m}^{i}M}{\mathfrak{m}^{i+1}M}$ be the image of $A_{i, d - (i+s)}$, and let
\[C_s = \bigoplus_{i=0}^s C_s^i.\]

We denote by $h_M$ the local Hilbert function of the module $M$, i.e.
\[ h_M(k) = \dim_{\mathbb{K}} \frac{\mathfrak{m}^k M}{\mathfrak{m}^{k+1} M} .\]

The theorem is stated as follows:

\begin{theorem}[symmetric decomposition]
\label{theorem-Iarrobino}
    If $R$ is a local ring and $M$ is a self-dual module, then the following statements hold true:
    \begin{enumerate}
        \item
            For all $0 \leq k, \: l \leq d + 1$, we have $Q_{k, \: 0} = Q_{d+1, \: l} = 0$, and thus $\Delta_{d+1-k}(k) = 0$.
        \item
        \label{point-upper-margin}
            For $k+l > d+1$ we have $A_{k, \: l} = A_{k, \: l-1}$ and so $Q_{k, \: l} = 0$.
        \item
            For $0 \leq a \leq d$ we have that $C_a$ is a submodule of $\Gr \: M$, and
            \[\frac{C_a}{C_{a+1}}\]
            is a self-dual module that is isomorphic to
            \[D_{a+1} = \bigoplus_{0 \leq t \leq d-a-1} Q_{t, \: d - (a+t)} = \bigoplus_{k+l = d-a} Q_{k, \: l}.\]
        \item 
            We have the $\mathbb{K}$-vector space isomorphism
            \[
            Q_{k, \: l}^{*} \cong Q_{l-1, \: k+1}.
            \]
            In particular, it holds that \[\Delta_s(t) = \Rank_{\mathbb{K}} \: Q_{t, \: d+1-(s+t)} = \Rank_{\mathbb{K}} \: Q_{t, \: d+1-(s+t)}^{*} = \Rank_{\mathbb{K}} \: Q_{d-s-t, \: t+1} = \Delta_s((d-s) - t).\]
            We say that $\Delta_s$ is symmetric with respect to $\frac{d-s}{2}$.
        \item
        \label{point-prefix}
            For $0 \leq a \leq d$ let us denote $h_a(t) = \sum_{i=0}^a \Delta_i(t)$. Then $h_a$ is the Hilbert function of the quotient
            \[\frac{\Gr \: M}{C_a}.\]
        \item
            The local Hilbert function $h_M$ satisfies $h_M(t) = 0$ for $t > d$ and
            \[h_M(t) = \sum_{i=0}^{d-t} \Delta_i(t) \text{ for } 0 \leq t \leq d.\]
            In particular, we have the decomposition from Equation \ref{equation-hilbert-decomposition}:
            \[
                h_M = \sum_{i=0}^d \Delta_i,
            \]
            where $\Delta_i$ is symmetric around $\frac{d-i}{2}$.
    \end{enumerate}
\end{theorem}

This theorem is proved in Section \ref{section-Iarrobino-s-decomposition}.
As examples of applications of this theorem, we show two claims that we also note in Section \ref{section-Classification}.
\begin{example}
\label{example-Iarrobino-s-decomposition}
    \begin{itemize}
        \item
        \label{example-class-first-last}
            If $h_M$ is a Hilbert function of a self-dual module $M$ and $k$ is the largest natural number such that $h_M(k) > 0$, then $h_M(0) \geq h_M(k)$.
        \begin{proof}
            Let us consider Iarrobino's symmetric decomposition of $h_M$ -- the first row must be symmetric from Theorem~\ref{theorem-Iarrobino} (\ref{point-prefix}) and must have $h_M(k)$ in the last position (since next rows are strictly shorter), so it also must have $h_M(k)$ in the first position.
            Since $h_M(0)$ is the sum of the values in the first positions in all rows, we have $h_M(0) \geq h_M(k)$.
        \end{proof}

        \item
        \label{example-class-symmetry-center}
            Let $h_M$ be a Hilbert function of a self-dual module $M$ and let $k$ be the largest natural number such that $h_M(k) > 0$. Then $\sum_{i = 0}^{\lfloor{\frac{k-1}{2}} \rfloor } h_M(i) \geq \sum_{i = \lceil{\frac{k+1}{2}} \rceil}^{k} h_M(i)$.
        \begin{proof}
            The claim states that the sum of the first half of the sequence is greater than or equal to the sum of the second half.
            This follows from Iarrobino's symmetric decomposition (Theorem~\ref{theorem-Iarrobino}) -- the sequence $h_M$ is the sum of the symmetric sequences, and each of these parts contributes to the first half at least as much as it does to the second half.
        \end{proof}
    \end{itemize}
\end{example}

J. Jelisiejew and K. \v{S}ivic introduced the concept of apolarity for modules, which is a very useful tool in generating examples of modules with given Hilbert function \cite{Jelisiejew2022ComponentsAS}.
M. Kunte introduced a sufficient and necessary condition for a module to be a graded self-dual module \cite{Kunte2008GorensteinMO}.
We use the notion of apolarity to prove a generalization of Kunte's criterion to the non-graded case. However, we show the criterion with only a sufficient part, while the criterion by Kunte for the graded case was also necessary.
The theorem is stated as follows:
Let $n$ be a positive natural number and let us denote $S^{*} = \mathbb{K}[x_1, \: x_2, \: \ldots, \: x_n]$.
We do not assume that the polynomials from $S^{*}$ are graded.
Let us also fix a positive natural number $r$.
We denote by $F^{*}$ the free $S^{*}$-module $\bigoplus_{k = 0}^r S^{*} e_k^{*}$.
Note that each $f_k \in F^{*}$ is encoded by the direct sum decomposition as $\sum_{i=0}^r f_{ki} e_i^*$, where $f_{ki} \in S^{*}$. Thus, the tuple $f_1, \: f_2, \ldots, f_l$ can be viewed as a matrix $N = [f_{ki}] \in \mathbb{M}_{l \times r}(S^{*})$. We call this matrix the encoding matrix of the module apolar to $(f_1, \: f_2, \ldots, f_l)$.

\begin{theorem}[Kunte's criterion]
\label{theorem-self-dual-square}
If the square matrix $N = [f_{ki}] \in \mathbb{M}_{r \times r}(S^{*})$ is symmetric, then the module apolar to $(f_k)_{k=1}^r$ is self-dual (in general non-graded), where $f_k \in F^*$.
\end{theorem}

A definition of the apolarity notion can be found in Section \ref{section-apolarity}.
Kunte's criterion is quite a simple condition, as we only need to verify if the proper matrix is symmetric.
This criterion, combined with apolarity for modules, allows us to easily generate examples of self-dual modules for a given Hilbert function.
The special case of this theorem is $r = 1$.
Then $M$ becomes an algebra.
Since the $1 \times 1$ matrix is trivially symmetric, it implies that every apolar algebra is self-dual.

\begin{example}
\label{example-kunte-diagonal}
    Let us take as the encoding matrix the diagonal matrix
        \[\begin{bmatrix}
            x^2 & 0 \\
            0 & x \\
        \end{bmatrix}.\]
    The obtained self-dual module has the local Hilbert function $(2, 2, 1)$. The method for computing the local Hilbert function from the encoding matrix is described in Section \ref{section-apolarity} (see Example~\ref{example-apolar-module}). The reader might observe that it corresponds to the direct sum decomposition $(1, 1, 1) + (1, 1)$ (see Section~\ref{section-Classification}).
\end{example}

We show other applications of this theorem in Example~\ref{example-apolar-module} and in the classification of the local Hilbert function $(2, 2, 3, 1)$ (see Section~\ref{section-Classification}).

B. Poonen classified the local Hilbert functions and the types of isomorphisms for self-dual algebras of degree less than or equal to six \cite{Poonen}.
The corresponding problem for modules is mostly open.
In our work, we classify possible local Hilbert functions of self-dual modules with a small degree -- we classify all possible Hilbert functions for a degree less than or equal to $8$.

\begin{proposition}
    The number of possible local Hilbert functions of self-dual modules for degree $m = 1, 2 \ldots 8$ is exactly the following:
\begin{table}[htbp]
    \caption{Number of possible local Hilbert functions of self-dual modules}
    \centering
    \begin{tabular}{lccl}
        \toprule
        Degree & Number of possible local Hilbert functions \\
        \midrule
        $m = 1$ & $1$ \\
        $m = 2$ & $2$ \\
        $m = 3$ & $3$ \\
        $m = 4$ & $6$ \\
        $m = 5$ & $9$ \\
        $m = 6$ & $16$ \\
        $m = 7$ & $24$ \\
        $m = 8$ & $38$ \\
        \bottomrule
    \end{tabular}
    \label{table:mr}
\end{table}
\end{proposition}

The whole classification of possible local Hilbert functions of self-dual modules can be found in Section~\ref{section-Classification}.
This classification is interesting, as we show that some unexpected local Hilbert functions are possible to obtain from self-dual modules, for example, the function $(2, 2, 3, 1)$.

\subsection{Self-dual algebras and modules}
\label{Subsection-self-dual}

Self-dual algebras (also known as Gorenstein algebras or commutative Frobenius algebras) are useful tools in algebraic complexity theory.
In this subsection, we specifically assume $\mathbb{K} = \mathbb{C}$ because the results in the references on tensors \cite{Jelisiejew2022ConciseTO, LaMM} are given only over $\mathbb{C}$.
Let us consider a tensor $t \in A \otimes_{\mathbb{C}} B \otimes_{\mathbb{C}} C$, where $A, \: B, \: C \cong \mathbb{C}^m$.
The tensor $t$ may be considered as a linear map $A^* \rightarrow B \otimes_{\mathbb{C}} C$.
We say that $t$ is $1_A$-generic if there exists an $\alpha$ such that $t(\alpha) : B^* \to C$ has full rank.
We define analogically $1_B$- and $1_C$-genericity.
We say that a tensor is $1$-generic, if it is $1_A$-, $1_B$- and $1_C$-generic.
Let us assume that $t$ is $1_A$-generic.
Let $\alpha$ be such an element that $t(\alpha)$ is invertible.
Let us consider the space $t(A^*) \cdot t(\alpha)^{-1}$.
We say that $t$ satisfies Strassen's equations if the space $t(A^*) \cdot t(\alpha)^{-1}$ is abelian, i.e., if it consists of commuting matrices.
More about Strassen's equations can be found in \cite{LaMM}.
A tensor is a structure tensor of an algebra $A$ if it corresponds to the multiplication map $A \times A \rightarrow A$.
The following result connects $1$-generic tensors satisfying Strassen's equations and self-dual algebras:
\begin{proposition}(\cite[Summary 2.5]{Jelisiejew2022ConciseTO})
\label{proposition-tensors-self-dual-algebras}
    Let us assume that we have $A, \: B, \: C \cong \mathbb{C}^m$.
    A tensor $t \in A \otimes_{\mathbb{C}} B \otimes_{\mathbb{C}} C$ is $1$-generic and satisfies Strassen's equations if and only if it is a structure tensor of a self-dual algebra.
\end{proposition}
The self-dual algebras are also a useful tool for investigating secant varieties, see \cite{MR3121848}.
They are subjects of intensive research, see \cite{MR2533305, casnati2015irreducibility}.

In this work, we investigate self-dual modules (see Definition \ref{definition-self-dual}) as they are also important in algebraic complexity theory.
Deriving algebraic structure on tensors allows us to analyze tensors of matrix multiplication and obtain bounds on $\omega$ \cite{strassen1969gaussian, Alman2021ARL, blser_et_al:LIPIcs:2020:12686, Wojtala2022}, which is one of the most important topics in theoretical computer science.
M. Wojtala introduced the notion of structure tensors of modules \cite{Wojtala2022}.
The following result connects $1_A$-generic tensors satisfying Strassen's equations and modules:
\begin{proposition}(\cite[Summary 2.5]{Jelisiejew2022ConciseTO}, \cite[Lemma 2.6]{LaMM}, \cite[Proof of Corollary~2]{Wojtala2022})
\label{proposition-tensors-modules}
    Let us assume that we have $A, \: B, \: C \cong \mathbb{C}^m$.
    A tensor $t \in A \otimes_{\mathbb{C}} B \otimes_{\mathbb{C}} C$ is $1_A$-generic and satisfies Strassen's equations if and only if it is a structure tensor of a module.
\end{proposition}
For researchers working with tensors, it is usual to impose symmetry conditions.
A tensor $t\in A\otimes_{\mathbb{C}} A\otimes_{\mathbb{C}} A$ is symmetric if it is invariant under the permutation of coordinates.
By Proposition~\ref{proposition-tensors-self-dual-algebras}, one cannot hope for symmetric $1_A$-generic tensors, which do not come from algebras.
However, imposing the partial symmetry is actually connected to coming from self-dual modules.
This connection is shown in the following proposition and can be viewed as an analogical result to Proposition~\ref{proposition-tensors-self-dual-algebras} and Proposition~\ref{proposition-tensors-modules}.
\begin{proposition}
\label{proposition-tensors-self-dual-modules}
    Let us assume that we have $A, \: B \cong \mathbb{C}^m$.
    If a tensor $t \in A \otimes_{\mathbb{C}} S^2(B)$ is $1_A$-generic and satisfies Strassen's equations, then it is a structure tensor of a self-dual module.
\end{proposition}
\begin{proof}
    Proposition~\ref{proposition-tensors-modules} implies that $t$ is a structure tensor of some module $M$.
    Now let us observe that by symmetry of matrices from $t(A^*)$ and a reasoning from \cite[Subsection 3.5]{Jelisiejew2022ComponentsAS} we have that $M$ is isomorphic to $M^*$ as modules.
\end{proof}

\subsection*{Acknowledgements}
{During the preparation of this publication, the author was part of the Szkoła Orłów program. The publication was created under the supervision of Joachim Jelisiejew, whose support and help were priceless. The author would like to thank the anonymous reviewer for helpful comments.}

\section{Preliminaries}

In this section, we recall facts that will be useful for our purposes.

\subsection{Modules and submodules}
\label{subsection-modules-and-submodules}

\begin{definition}
\label{definition-perp}
    For a submodule $K$ of $M$ we define $K^{\perp}$ by the formula $K^{\perp} = \left\{ \phi \in M^{*} : \phi(K) = 0 \right\}$.
\end{definition}

\begin{lemma}
\label{remark-perp-linspace}
    For a module $K$ we have that $K^{\perp}$ is a submodule of $M^{*}$.
\end{lemma}
\begin{proof}
    If $\phi_1, \phi_2 \in K^{\perp}$, then $\phi_1(K) = 0, \phi_2(K) = 0$, so $(\phi_1 + \phi_2)(K) = \phi_1(K) + \phi_2(K) = 0$, so $\phi_1 + \phi_2 \in K^{\perp}$.
    Also if $\phi \in K^{\perp}$, then for every $r \in R$ we have that $rK \subseteq K$ (since $K$ is a submodule of $M$) and thus $(r \cdot \phi) (K) = \phi(rK) = 0$, so $r \cdot \phi \in K$.
\end{proof}

\begin{lemma}
\label{lemma-obvious}
    \begin{enumerate}
        \item
        \label{lemma-dim_perp}
            Let $K$ be a submodule of $M$. Then $\dim_{\mathbb{K}} K + \dim_{\mathbb{K}} K^{\perp} =        \dim_{\mathbb{K}} M$.
        \item
        \label{lemma-double_perp}
            Let $K$ be a submodule of $M$. Then $(K^{\perp})^{\perp} \cong K$.
        \item
        \label{lemma-Sum}
            Let $K, \: L$ be submodules of $M$. Then we have $(K+L)^{\perp} = K^{\perp} \cap L^{\perp}$.
        \item
        \label{lemma-Cap}
            Let $K, \: L$ be submodules of $M$. Then $(K \cap L)^{\perp} = K^{\perp} + L^{\perp}$.
    \end{enumerate}
\end{lemma}
\begin{proof}
    Left to the reader.
\end{proof}

\begin{lemma}
\label{lemma-bidual-module}
    We have a canonical isomorphism $M \cong M^{**}$ given by the double dual map.
\end{lemma}
\begin{proof}
    Let $\psi: M \rightarrow M^{**}$ be the double dual map, i.e. $\psi(m)(\phi) = \phi(m)$.
    Then $\psi$ is an isomorphism of $\mathbb{K}$-linear spaces, we need to show that it is also an isomorphism of $R$-modules.
    Indeed, we have $s \cdot \psi(m)(\phi) = \psi(m)(s \phi) = (s \phi)(m) = \phi(sm) = \psi(s \cdot m)(\phi)$.
\end{proof}

\begin{lemma}
\label{lemma-Dual}
    Let $I$ be a submodule of $M$. Then we have a canonical isomorphism $I^{*} \cong \sfrac{M^{*}}{I^{\perp}}$ induced by the canonical surjection.
\end{lemma}
\begin{proof}
    Let $\pi : M^{*} \rightarrow I^{*}$ be the canonical surjection. Then $\ker \pi = \left\{ \phi \in M^{*} : \phi(I) = 0 \right\} = I^{\perp}$.
    So by the isomorphism theorem we have a canonical isomorphism $I^{*} \cong \sfrac{M^{*}}{\ker \pi} = \sfrac{M^{*}}{I^{\perp}}$.
\end{proof}

\begin{lemma}
\label{lemma-Quotient}
    Let $J$ be a submodule of $I$ and $I$ be a submodule of $M$.
    Then we have a canonical exact sequence $0 \rightarrow (\sfrac{I}{J})^{*} \rightarrow I^{*} \rightarrow J^{*} \rightarrow 0$.
\end{lemma}
\begin{proof}
    We obtain the canonical injection $(\sfrac{I}{J})^{*} \rightarrow I^{*}$ by lifting functionals from $(\sfrac{I}{J})^{*}$ to functionals from $I^{*}$, i.e. functional $\phi$ goes to functional $\overline{\phi}$, such that for $i \in I$, $i \equiv \overline{i} \mod J$ we have $\overline{\phi}(i) = \phi(\overline{i})$. The image of this injection consists of all functionals that are zero on $J$. The canonical surjection $I^{*} \rightarrow J^{*}$ is obtained by restricting functionals on $I$ to $J$, its kernel also consists of all functionals that are zero on $J$, so we constructed the desired canonical exact sequence.
\end{proof}

\begin{lemma}
\label{lemma-Frac_perp}
    Let $J$ be a submodule of $I$ and $I$ be a submodule of $M$.
    Then $(\sfrac{I}{J})^{*} \cong \sfrac{J^{\perp}}{I^{\perp}}$.
\end{lemma}
\begin{proof}
    Left to the reader.
\end{proof}

\begin{lemma}
\label{lemma-perfect-pairing}
    Let $M_1, \: M_2$ be $R$-modules that are finite-dimensional $\mathbb{K}$-vector spaces.
    Let us assume that there exists a non-degenerate pairing $\langle \: , \: \rangle: M_1 \times M_2 \rightarrow \mathbb{K}$
    such that for every $r \in R, \: a_1 \in M_1, \: a_2 \in M_2$, we have $\langle r a_1, \: a_2 \rangle = \langle a_1, \: r a_2 \rangle$.
    Then $M_1$ and $M_2$ are dual as $R$-modules.
\end{lemma}
\begin{proof}
    Left to the reader.
\end{proof}

\begin{definition}
\label{definition-annihilator}
    For an ideal $\mathfrak{a}$ in $R$ we denote by $(0 : \mathfrak{a})_M$ the annihilator of $\mathfrak{a}$ in $M$, i.e. the set of all elements $x \in M$, such that $\mathfrak{a} x = 0$. It is easy to check that $(0 : \mathfrak{a})_M$ is a submodule of $M$.
\end{definition}

\begin{lemma}
\label{lemma-Annihilator}
    Let $\mathfrak{a}$ be an ideal in $R$.
    If $M$ is self-dual, then we have the following isomorphism of $R$-modules: $(\mathfrak{a} M)^{\perp} \cong (0 : \mathfrak{a})_M$.
\end{lemma}
\begin{proof}
    We have $(\mathfrak{a} M)^{\perp} = \left\{ \phi \in M^{*} : \phi(\mathfrak{a} M) = 0 \right\} = \left\{ \phi \in M^{*} : \forall_{s \in \mathfrak{a}} {s \phi = 0} \right\} \cong \left\{ l \in M : \forall_{s \in \mathfrak{a}} s l = 0 \right\} = (0 : \mathfrak{a})_M$, where the isomorphism follows from $M \cong M^*$.
\end{proof}

\subsection{Local rings}
\label{subsection-Local-rings}
In this subsection, we assume that $R$ is a local ring and $\mathfrak{m}$ is its only maximal ideal.

\begin{corollary}
\label{corollary-Annihilator_m}
    If $M$ is self-dual, then we have the following isomorphisms of $R$-modules:
    $\left( \mathfrak{m}^k M \right) ^{\perp} \cong (0 : \mathfrak{m}^k M)_M$,
    $\left( (0 : \mathfrak{m}^k M)_M \right) ^{\perp} \cong \mathfrak{m}^k M$.
\end{corollary}
\begin{proof}
    The first statement is just Lemma \ref{lemma-Annihilator} for the ideal $\mathfrak{m}^k$. The second part is obtained from the first one by taking the perpendicular submodule on both sides and using Lemma~\ref{lemma-obvious} (\ref{lemma-double_perp}).
\end{proof}

Let us recall the Remark \ref{lemma-finite-Hilbert}.

\subsection{Modularity and filtrations}
\label{subsection-Modularity-and-Filtrations}

\begin{remark}
\label{lemma-Modularity}
    Let $A, \: B, \: C$ be submodules of $M$ and let us assume that $A$ is a submodule of $C$.
    Then $(A+B) \cap C = B \cap C + A$.
\end{remark}

\begin{corollary}
\label{corollary-Modularity_pairs}
    Let $A_1, \: A_2, \: B_1, \: B_2$ be submodules of $M$ and let us also assume that $A_1$ is a submodule of $B_1$ and $B_2$ is a submodule of $A_2$.
    Then $(A_1 + A_2) \cap (B_1 + B_2) = A_1 + B_2 + A_2 \cap B_1$.
\end{corollary}
\begin{proof}
    First, we use Remark~\ref{lemma-Modularity} for $A = A_1, \: B = A_2, \: C = (B_1 + B_2)$ and obtain $(A_1 + A_2) \cap (B_1 + B_2) = A_1 + A_2 \cap (B_1 + B_2)$.
    Then we use Remark~\ref{lemma-Modularity} for $A = B_2, \: B = B_1, \: C = A_2$ and obtain $A_2 \cap (B_1 + B_2) = B_2 + B_1 \cap A_2$.
\end{proof}

\begin{remark}
\label{lemma-filtration_quotient}
    Let $F_0 \subseteq F_1 \subseteq \ldots \subseteq F_c$ and $G_0 \subseteq G_1 \subseteq \ldots \subseteq G_c$ be filtrations of $R$-modules, such that $F_k \subseteq G_k$ and $G_k \cap F_{k+1} = F_k$ for $k = 0, \: 1, \: \ldots, \: c-1$. Then we have the filtration $\frac{G_0}{F_0} \subseteq \frac{G_1}{F_1} \subseteq \ldots \subseteq \frac{G_c}{F_c}$ of $R$-modules.
\end{remark}

\subsection{Associated graded ring and module}
\label{subsection-Associated-graded-ring-and-module}

In this subsection, we assume that $R$ is a local ring and $\mathfrak{m}$ is its (only) maximal ideal.
We denote by $d$ the socle degree of $M$, as in Remark \ref{lemma-finite-Hilbert}.

\begin{definition}
    By $\Gr R$ we denote the associated graded ring of $R$, i.e.
    \[ \Gr R = \bigoplus_{k=0}^d \frac{\mathfrak{m}^k}{\mathfrak{m}^{k+1}}. \]
    The multiplication in this ring is defined as follows:
    for $a_k \in \mathfrak{m}^k$ let $\overline{a_k}$ denote the image of $a_k$ in $\frac{\mathfrak{m}^k}{\mathfrak{m}^{k+1}}$; then $\overline{a_i} \cdot \overline{a_j}$ is defined as $\overline{a_i \cdot a_j}$, i.e. it is the image of $a_i \cdot a_j$ in $\frac{\mathfrak{m}^{i+j}}{\mathfrak{m}^{i+j+1}}$.
    
    By $\Gr M$ we denote the associated graded module of the module $M$, i.e.
     \[ \Gr M = \bigoplus_{k=0}^d \frac{\mathfrak{m}^k M}{\mathfrak{m}^{k+1} M}, \]
     which is a ($\Gr R$)-module in a natural way.
\end{definition}

\begin{definition}
    We denote by $h_M$ the local Hilbert function of the module $M$, i.e.
    \[ h_M(k) = \dim_{\mathbb{K}} \frac{\mathfrak{m}^k M}{\mathfrak{m}^{k+1} M} .\]
\end{definition}

Let us denote $\mathfrak{n} = \bigoplus_{k > 0} \frac{\mathfrak{m}^{k}}{\mathfrak{m}^{k+1}} \subseteq \Gr R$.
Clearly $\mathfrak{n}$ is an ideal in $\Gr R$. The following lemma shows that it is the unique maximal ideal in $\Gr R$.

\begin{lemma}
\label{lemma-gradation_locality}
    The associated graded ring $\Gr R$ is a local ring and $\mathfrak{n}$ is its unique maximal ideal.
\end{lemma}
\begin{proof}
    We have \[\frac{\Gr R}{\mathfrak{n}} = \frac{\Gr R}{\bigoplus_{k > 0} \frac{\mathfrak{m}^{k}}{\mathfrak{m}^{k+1}}} \cong \frac{R}{\mathfrak{m}}.\] Since $\frac{R}{\mathfrak{m}}$ is a field, we obtain that $\mathfrak{n}$ is a maximal ideal.
    Moreover, $\mathfrak{n}^{d+1} = (\bigoplus_{k > 0} \frac{\mathfrak{m}^{k}}{\mathfrak{m}^{k+1}})^{d+1} = 0$, so since maximal ideals are prime, every maximal ideal must contain $\mathfrak{n}$.
    It follows that $\mathfrak{n}$ is the only maximal ideal in $\Gr R$.
\end{proof}

\begin{lemma}
\label{lemma-gradation_maximal_power}
    It holds that $\mathfrak{n}^{l} = \bigoplus_{k \geq l} \frac{\mathfrak{m}^{k}}{\mathfrak{m}^{k+1}}$.
\end{lemma}
\begin{proof}
    Clearly $\mathfrak{n}^{l} = (\bigoplus_{k \geq 1} \frac{\mathfrak{m}^{k}}{\mathfrak{m}^{k+1}})^l \subseteq \bigoplus_{k \geq l} \frac{\mathfrak{m}^{k}}{\mathfrak{m}^{k+1}}$.
    To show the opposite inclusion, let us fix an element $a \in \frac{\mathfrak{m}^{l+c}}{\mathfrak{m}^{l+c+1}}$, where $c \geq 0$.
    We can write $a$ as a finite sum $\sum_j a_{1, \: j} \cdot a_{2, \: j} \cdot \ldots \cdot a_{l+c, \: j} + \mathfrak{m}^{l+c+1}$, where $a_{i, \: j} \in \mathfrak{m}$.
    Then we have $a = \sum_j a_{1, \: j} \cdot a_{2, \: j} \cdot \ldots \cdot a_{l+c, \: j} + \mathfrak{m}^{l+c+1} = \sum_j (a_{1, \: j} + \mathfrak{m}^2) \cdot (a_{2, \: j} + \mathfrak{m}^2) \cdot \ldots \cdot (a_{l+c, \: j} + \mathfrak{m}^2) \in (\frac{\mathfrak{m}}{\mathfrak{m}^2})^{l+c} \subseteq (\bigoplus_{k \geq 1} \frac{\mathfrak{m}^{k}}{\mathfrak{m}^{k+1}})^{l+c} \subseteq (\bigoplus_{k \geq 1} \frac{\mathfrak{m}^{k}}{\mathfrak{m}^{k+1}})^{l}$.
\end{proof}

\begin{corollary}
\label{corollary-gradation_Hilbert}
    Let $C = \bigoplus_{k=0}^{d} \frac{C_k + \mathfrak{m}^{k+1} M}{\mathfrak{m}^{k+1} M}$ be a graded submodule of $\Gr M$, where $C_k \subseteq \mathfrak{m}^{k} M$ is an $R$-submodule.
    Then the local Hilbert function $h_q$ of the quotient $\frac{\Gr M}{C}$ satisfies $h_q(k) = \dim_{\mathbb{K}} \frac{\mathfrak{m}^{k} M}{\mathfrak{m}^{k+1} M + C_k}$
    and the local Hilbert function $h_{C}$ of the module $C$ satisfies $h_{C}(k) = \dim_{\mathbb{K}} \frac{C_k + \mathfrak{m}^{k+1} M}{\mathfrak{m}^{k+1} M}$.
\end{corollary}
\begin{proof}
    by Lemma~\ref{lemma-gradation_maximal_power} it follows that
    \[\frac{\mathfrak{n}^{l}}{\mathfrak{n}^{l+1}} = \frac{\bigoplus_{k \geq l} \frac{\mathfrak{m}^{k}}{\mathfrak{m}^{k+1}}}{\bigoplus_{k \geq l+1} \frac{\mathfrak{m}^{k}}{\mathfrak{m}^{k+1}}} = \frac{\mathfrak{m}^{l}}{\mathfrak{m}^{l+1}}.\]
    So the direct sum decompositions of $\Gr M$ coming from $(\frac{\mathfrak{n}^{l}}{\mathfrak{n}^{l+1}})_{l=0}^d$ and $(\frac{\mathfrak{m}^{l}}{\mathfrak{m}^{l+1}})_{l=0}^d$ are identical, so values of the Hilbert function can be obtained from the latter one, which we wanted to show.
\end{proof}

The following remark states that the associated graded ring $\Gr R$ satisfies the same assumptions as $R$ and $\Gr M$ satisfies the same assumptions as $M$ (we do not assume that $M$ is self-dual here). It implies that we can apply Macaulay's Bound to $\Gr M$. Note that $\Gr M$ in general might not be self-dual, so we cannot apply Iarrobino's symmetric decomposition to $\Gr M$.

\begin{remark}
\label{lemma-gradation}
    If $R$ is a local ring, then the associated graded ring $\Gr R$ is a local $\mathbb{K}$-algebra and is a finite-dimensional $\mathbb{K}$-vector space.
    The associated graded module $\Gr M$ is a finite-dimensional $\mathbb{K}$-vector space.
\end{remark}

\subsection{Macaulay's Bound}
\label{subsection-Macaulay-s-Bound}

In this subsection, we assume that $R$ is a local ring.
We recall now a criterion (thanks to Macaulay) that the local Hilbert functions of modules have to satisfy.
Unlike the symmetric Iarrobino's decomposition, this criterion is also valid for modules that are not self-dual.
Before we recall the aforesaid Macaulay's Bound, we need to introduce a representation system.

\begin{definition}
\label{definition-Macaulay}
    Let $a, \: s$ be positive integers. Then there exist uniquely determined positive integers $a_s > a_{s-1} > a_{s-2} > \ldots > a_k \geq k > 0$, such that we have the equality
    \[a = \binom{a_s}{s} + \binom{a_{s-1}}{s-1} + \binom{a_{s-2}}{s-2} + \ldots + \binom{a_k}{k}.\]
    We denote
    \[a^{\langle s \rangle} = \binom{a_s + 1}{s+1} + \binom{a_{s-1} + 1}{s} + \binom{a_{s-2} + 1}{s-1} + \ldots + \binom{a_k + 1}{k+1}.\]
\end{definition}

\begin{theorem}[Macaulay's Bound, \cite{Hulett}]
\label{theorem-Macaulay}
    We have the following bound on the local Hilbert function:
    \[h_M(r + 1) \leq h_M(r)^{\langle r \rangle}\]
    for $r \geq 1$.
    The bound holds true for all finite-length modules $M$ (even not self-dual).
\end{theorem}

\begin{corollary}
\label{corollary-Macaulay_d}
    Let us assume that it holds $h_M(r) \leq r$. Then we have $h_M(r+1) \leq h_M(r)$.
\end{corollary}
\begin{proof}
    If $r = 0$, then the assumption states that $h_M(0) = 0$, so $M = 0$ and the claim holds true.
    Now let us assume that $r \geq 1$.
    Since we have $h_M(r) \leq r$, the representation of $h_M(r)$ is $\binom{r}{r} + \binom{r-1}{r-1} + \binom{r-2}{r-2} + \ldots + \binom{r- h_M(r) + 1}{r - h_M(r) + 1}$. Thus $h_M(r)^{\langle r \rangle} = h_M(r)$ and by Macaulay's Bound \ref{theorem-Macaulay} we obtain $h_M(r+1) \leq h_M(r)^{\langle r \rangle} = h_M(r)$.
\end{proof}

\begin{corollary}
\label{corollary-Macaulay_d_plus_1}
     Let us assume that it holds $h_M(r) \leq r+1$ and $r \geq 1$. Then we have $h_M(r+1) \leq r+2$.
\end{corollary}
\begin{proof}
    If $h_M(r) \leq r$, then from Corollary~\ref{corollary-Macaulay_d} we obtain the claim.
    If $h_M(r) = r+1$, then the representation of $h_M(r)$ is $\binom{r+1}{r}$ and thus $h_M(r)^{\langle r \rangle} = r+2$. So by Macaulay's Bound \ref{theorem-Macaulay} we obtain $h_M(r+1) \leq h_M(r)^{\langle r \rangle} = r+2$.
\end{proof}

\section{Iarrobino's symmetric  decomposition}
\label{section-Iarrobino-s-decomposition}

In this section, we introduce a new result -- Iarrobino's symmetric decomposition for modules.
Iarrobino's symmetric  decomposition was known for algebras, but as far as we know, this is the first time that Iarrobino's decomposition has been considered for modules.
With properly defined objects, the generalization from algebras to modules is direct, but there are some subtle differences, see Example~\ref{example-Q_0_l}.
We assume in this section that $R$ is a local ring and $\mathfrak{m}$ is its only maximal ideal.
Additionally, we assume that $M$ is self-dual, i.e. $M$ and $M^{*}$ are isomorphic as $R$-modules.

Let $d$ be a maximal natural number such that $\mathfrak{m}^{d}M \neq 0$ ($d$ exists, as we observe in Remark~\ref{lemma-finite-Hilbert}).
Let us denote 
\[
A_{k, \: l} = \mathfrak{m}^{k}M \cap (0 : \mathfrak{m}^{l})_M \text{ for all integers } k, \: l \text{ (we follow the convention from Definition \ref{definition-negative-powers})},
\]
\begin{equation}
\label{equation-q-def}
    Q_{k, \: l} = \frac{A_{k, \: l}}{A_{k + 1, \: l} + A_{k, \: l - 1}} \text{ for all } 0 \leq k, \: l \leq d + 1.
\end{equation}
Let us also define
\[
\Delta_s(t) = 
   \begin{cases*}
      \Rank_{\mathbb{K}} \: Q_{t, \: d+1-(s+t)} & for $0 \leq s \leq d, \: 0 \leq t \leq d-s$ \\
      0           & otherwise.
    \end{cases*}
\]

Let us recall that Remark~\ref{remark-Q-delta} gives us the inverse formula for $\Rank_{\mathbb{K}}Q_{k, \: l}$.

For $0 \leq s \leq d$, let us define the vector space
\[D_s = \bigoplus_{0 \leq t \leq d-s} Q_{t, \: d+1 - (s+t)} = \bigoplus_{k+l = d+1-s} Q_{k, \: l}.\]

For $0 \leq s \leq d$ let $C_s^i \subseteq \sfrac{\mathfrak{m}^{i}M}{\mathfrak{m}^{i+1}M}$ be the image of $A_{i, d - (i+s)}$, and let
\[C_s = \bigoplus_{i=0}^s C_s^i.\]

\begin{proposition}
\label{proposition-Identity}
    Let $R$ be a local ring and $M$ be a self-dual module.
    Then, we have the following isomorphism of $\mathbb{K}$-vector spaces
    \[ \left( \frac{\mathfrak{m}^iM \cap (0 : \mathfrak{m}^j)_M}{\mathfrak{m}^{i+1}M \cap (0 : \mathfrak{m}^j)_M + \mathfrak{m}^iM \cap (0 : \mathfrak{m}^{j-1})_M} \right)^{*} \cong \frac{(0 : \mathfrak{m}^{i+1})_M \cap \mathfrak{m}^{j-1}M} {(0 : \mathfrak{m}^i)_M \cap \mathfrak{m}^{j-1}M + (0 : \mathfrak{m}^{i+1})_M \cap \mathfrak{m}^jM}. \]
    Equivalently, we have the following isomorphism of $\mathbb{K}$-vector spaces
    \[Q_{k, \: l}^{*} \cong Q_{l-1, \: k+1}.\]
\end{proposition}
\begin{proof}
    Using Lemma~\ref{lemma-Frac_perp}, Lemma \ref{lemma-obvious} (\ref{lemma-Sum}) and Lemma~\ref{lemma-obvious} (\ref{lemma-Cap}) respectively, we obtain
    \begin{equation}
        \begin{aligned}
            \left( \frac
            {\mathfrak{m}^iM \cap (0 : \mathfrak{m}^j)_M}
            {\mathfrak{m}^{i+1}M \cap (0 : \mathfrak{m}^j)_M + \mathfrak{m}^iM \cap (0 : \mathfrak{m}^{j-1})_M}
            \right)^{*} \cong \frac{
            \left( \mathfrak{m}^{i+1}M \cap (0 : \mathfrak{m}^j)_M + \mathfrak{m}^iM \cap (0 : \mathfrak{m}^{j-1})_M \right)^{\perp}}
            {\left( \mathfrak{m}^iM \cap (0 : \mathfrak{m}^j)_M \right)^{\perp}} \cong \\ \cong \frac{
            \left( \mathfrak{m}^{i+1}M \cap (0 : \mathfrak{m}^j)_M \right)^{\perp} \cap \left( \mathfrak{m}^iM \cap (0 : \mathfrak{m}^{j-1})_M \right)^{\perp}}
            {\left( \mathfrak{m}^iM \cap (0 : \mathfrak{m}^j)_M \right)^{\perp}} \cong \\ \cong \frac{
            \left( \left( \mathfrak{m}^{i+1}M \right)^{\perp} + \left( (0 : \mathfrak{m}^j)_M \right)^{\perp} \right) \cap \left( \left( \mathfrak{m}^iM \right)^{\perp} + \left( (0 : \mathfrak{m}^{j-1})_M \right)^{\perp} \right)}
            {\left( \mathfrak{m}^iM \right)^{\perp} + \left( (0 : \mathfrak{m}^j)_M \right)^{\perp}}.
        \end{aligned}
    \end{equation}
    
    Now from Corollary~\ref{corollary-Annihilator_m} we have
        \begin{equation}
        \begin{aligned}
            \frac{
            \left( \left( \mathfrak{m}^{i+1}M \right)^{\perp} + \left( (0 : \mathfrak{m}^j)_M \right)^{\perp} \right) \cap \left( \left( \mathfrak{m}^iM \right)^{\perp} + \left( (0 : \mathfrak{m}^{j-1})_M \right)^{\perp} \right)}
            {\left(  \mathfrak{m}^iM \right)^{\perp} + \left( (0 : \mathfrak{m}^j)_M \right)^{\perp}}  \cong \\ \cong \frac
            {\left( (0 : \mathfrak{m}^{i+1})_M + \mathfrak{m}^jM \right) \cap \left( (0 : \mathfrak{m}^i)_M + \mathfrak{m}^{j-1}M \right)}
            {(0 : \mathfrak{m}^i)_M + \mathfrak{m}^jM}.
        \end{aligned}
    \end{equation}
    
    From Corollary~\ref{corollary-Modularity_pairs} we obtain
    \begin{equation}
        \begin{aligned}
            \frac
            {\left( (0 : \mathfrak{m}^{i+1})_M + \mathfrak{m}^jM \right) \cap \left( (0 : \mathfrak{m}^i)_M + \mathfrak{m}^{j-1}M \right)}
            {(0 : \mathfrak{m}^i)_M + \mathfrak{m}^jM} \cong \\ \cong \frac
            {(0 : \mathfrak{m}^i)_M + \mathfrak{m}^jM + (0 : \mathfrak{m}^{i+1})_M \cap \mathfrak{m}^{j-1}M}
            {(0 : \mathfrak{m}^i)_M + \mathfrak{m}^jM}.
        \end{aligned}
    \end{equation}
    Now we can rewrite
        \begin{equation}
        \begin{aligned}
            \frac
            {(0 : \mathfrak{m}^i)_M + \mathfrak{m}^jM + (0 : \mathfrak{m}^{i+1})_M \cap \mathfrak{m}^{j-1}M}
            {(0 : \mathfrak{m}^i)_M + \mathfrak{m}^jM} \cong \\ \cong \frac
            {(0 : \mathfrak{m}^{i+1})_M \cap \mathfrak{m}^{j-1}M}
            {\left( (0 : \mathfrak{m}^i)_M + \mathfrak{m}^jM \right) \cap  \left( (0 : \mathfrak{m}^{i+1})_M \cap \mathfrak{m}^{j-1}M \right)} \cong \\ \cong \frac
            {(0 : \mathfrak{m}^{i+1})_M \cap \mathfrak{m}^{j-1}M}
            {(0 : \mathfrak{m}^i)_M \cap \mathfrak{m}^{j-1}M + (0 : \mathfrak{m}^{i+1})_M \cap \mathfrak{m}^{j}M}.
        \end{aligned}
    \end{equation}
\end{proof}

\begin{proposition}
\label{proposition_D_self-dual}
    $D_s$ is a self-dual module over the associated graded ring $\Gr R$.
\end{proposition}
\begin{proof}
    First, we prove that $D_s$ is indeed a module.
    For a given $r \in \mathfrak{m}^i$, multiplication by $r$ induces a linear function from $A_{k, \: l}$ to $A_{k+i, \: l-i}$.
    After dividing by $A_{k+1, \: l} + A_{k, \: l-1}$, it induces a linear function from $Q_{k, \: l}$ to $Q_{k+i, \: l-i}$.
    Moreover, multiplication by $\mathfrak{m}^{i+1}$ acts zero on $Q_{k, \: l}$, so we obtain a linear function from $\frac{\mathfrak{m}^i}{\mathfrak{m}^{i+1}} \times Q_{k, \: l}$ to $Q_{k+i, \: l-i}$.
    Summing over $i$, we obtain a linear function from $\Gr R \times Q_{k, \: l}$ to $\bigoplus_{i \geq 0} Q_{k+i, \: l-i}$.
    Let us fix $s$.
    Summing over $k+l = d+1-s$, we obtain a linear function from $\Gr R \times D_s$ to $D_s$.
    It is straightforward that this function satisfies the module axioms as it comes from multiplication by $r$.
    Proposition~\ref{proposition-Identity} states that $Q_{k, \: l}^{*} \cong Q_{l-1, \: k+1}$ as $\mathbb{K}$-vector spaces.
    Summing over $k+l = d+1-s$, we obtain the isomorphism $D_s \cong D_s^*$ of the $\mathbb{K}$-vector spaces.
    Moreover, the isomorphism from Proposition~\ref{proposition-Identity} is consistent with multiplication (in the proof, we use two isomorphisms: the canonical isomorphism from Lemma~\ref{lemma-Frac_perp} and the $R$-modules isomorphism from Corollary~\ref{corollary-Annihilator_m}, both of which are consistent with multiplication), so we obtain the $\Gr R$-modules isomorphism.
\end{proof}

\begin{theorem-non}[Theorem~\ref{theorem-Iarrobino} -- symmetric decomposition]
\label{theorem-non-Iarrobino}
    If $R$ is a local ring and $M$ is a self-dual module, then the following statements hold true:
    \begin{enumerate}
        \item
            For all $0 \leq k, \: l \leq d + 1$, we have $Q_{k, \: 0} = Q_{d+1, \: l} = 0$, and thus $\Delta_{d+1-k}(k) = 0$.
        \item
        \label{non-point-upper-margin}
            For $k+l > d+1$ we have $A_{k, \: l} = A_{k, \: l-1}$ and so $Q_{k, \: l} = 0$.
        \item
            For $0 \leq a \leq d$ we have that $C_a$ is a submodule of $\Gr \: M$, and
            \[\frac{C_a}{C_{a+1}}\]
            is a self-dual module over the associated graded ring $\Gr R$ and is isomorphic to
            \[D_{a+1} = \bigoplus_{0 \leq t \leq d-a-1} Q_{t, \: d - (a+t)} = \bigoplus_{k+l = d-a} Q_{k, \: l}.\]
        \item 
            We have the $\mathbb{K}$-vector space isomorphism
            \[
            Q_{k, \: l}^{*} \cong Q_{l-1, \: k+1}.
            \]
            In particular, it holds that \[\Delta_s(t) = \Rank_{\mathbb{K}} \: Q_{t, \: d+1-(s+t)} = \Rank_{\mathbb{K}} \: Q_{t, \: d+1-(s+t)}^{*} = \Rank_{\mathbb{K}} \: Q_{d-s-t, \: t+1} = \Delta_s((d-s) - t).\]
            We say that $\Delta_s$ is symmetric with respect to $\frac{d-s}{2}$.
        \item
        \label{non-point-prefix}
            For $0 \leq a \leq d$ let us denote $h_a(t) = \sum_{i=0}^a \Delta_i(t)$. Then $h_a$ is the Hilbert function of the quotient
            \[ \frac{\Gr \: M}{C_a}. \]
        \item
            The local Hilbert function $h_M$ satisfies $h_M(t) = 0$ for $t > d$ and
            \[h_M(t) = \sum_{i=0}^{d-t} \Delta_i(t) \text{ for } 0 \leq t \leq d.\]
            In particular, we have the decomposition from Equation \ref{equation-hilbert-decomposition}:
            \[
                h_M = \sum_{i=0}^d \Delta_i,
            \]
            where $\Delta_i$ is symmetric around $\frac{d-i}{2}$.
    \end{enumerate}
\end{theorem-non}

\begin{proof}[Proof of Theorem~\ref{theorem-Iarrobino}]
    \begin{enumerate}
        \item
            We have $(0 : \mathfrak{m}^{0})_M = (0 : \mathfrak{m}^{-1})_M$, so $A_{k, \: 0} = A_{k, \: -1}$ and thus $Q_{k, \: 0} = 0$. Also, since $\mathfrak{m}^{d+1} = 0$, we have that $A_{d+1, \: l} = 0$, so $Q_{d+1, \: l} = 0$.
            From Remark \ref{remark-Q-delta} we deduce $\Delta_{d+1-k}(k) = 0$.
        \item
            Clearly, we have $A_{k, \: l} \subseteq A_{k, \: l-1}$, so let us prove the other inclusion. Let us assume that $x \in A_{k, \: l-1}$. Then $x \in \mathfrak{m}^kM$, so $\mathfrak{m}^{l-1} x \in \mathfrak{m}^{k+l-1}M$. Since by the assumption $k+l-1 \geq d+1$, it holds that $\mathfrak{m}^{k+l-1} = 0$, so $\mathfrak{m}^{l-1} x = 0$. That proves that $x \in (0 : \mathfrak{m}^{l-1})_M$, so indeed $A_{k, \: l-1} \subseteq A_{k, \: l}$.
        \item
            First, we need to argue that $C_a$ is actually a submodule of $\Gr \: M$.
            Fix $i$.
            We have that $C_a^i$ is the image of $A_{i, d - (i+a)}$.
            For any $k, \: i \geq 0$ we have
            \begin{align*}
                \mathfrak{m}^k \cdot A_{i, d - (i+a)} &= \mathfrak{m}^k \cdot (\mathfrak{m}^{i}M \cap (0:\mathfrak{m}^{d - (i+a)})_M) = \mathfrak{m}^{k+i}M \cap (\mathfrak{m}^k \cdot (0:\mathfrak{m}^{d - (i+a)})_M) \subseteq \\
                &\subseteq \mathfrak{m}^{k+i}M \cap (0:\mathfrak{m}^{d - (i+k+a)})_M =  A_{i+k, d - ((i+k)+a)},
            \end{align*}
            so indeed we obtain a submodule of the associated graded module.
            
            Denote $e = d - (k+a)$.            
            We have the filtration
            \[A_{k, \: e} \subseteq A_{k, \: e+1} \subseteq A_{k, \: e+2} \subseteq  \ldots \subseteq A_{k, \: d} \subseteq A_{k, \: d+1} = \mathfrak{m}^{k}M,\]
            which induces the filtration
            \[C_{d}^k \subseteq C_{d-1}^k \subseteq C_{d-2}^k \subseteq \ldots \subseteq C_{1}^k \subseteq C_{0}^k,\]
            and this gives us
            \[C_{d}\subseteq C_{d-1} \subseteq C_{d-2} \subseteq \ldots \subseteq C_{1} \subseteq C_{0}.\]

            We have 
            \[C_a^k \cong \frac{A_{k, \: e}}{\mathfrak{m}^{k+1} \cap A_{k, \: e}} \cong \frac{A_{k, \: e}}{A_{k+1, \: e}}.\]
            Then,
            \[\frac{C_a^k}{C_{a+1}^k} \cong \frac{\frac{A_{k, \: e}}{A_{k+1, \: e}}}{\frac{A_{k, \: e-1}}{A_{k+1, \: e-1}}}.\]
            We have $A_{k+1, \: e-1} \subseteq A_{k+1, \: e}$, which gives us
            \[\frac{C_a^k}{C_{a+1}^k} \cong \frac{A_{k, \: e}}{A_{k, \: e-1} + A_{k+1, \: e-1}} = Q_{k, \: e}.\]
            Thus, 
            \[\frac{C_a}{C_{a+1}} \cong \bigoplus_{k = 0}^a Q_{k, \: d-k-a} = D_{s+1}.\]
            These isomorphisms are canonical, and from Proposition~\ref{proposition_D_self-dual} we know that $D_{s+1}$ is a self-dual module, so we obtain that $\frac{C_a}{C_{a+1}}$ is a self-dual module isomorphic to $D_{s+1}$.

        \item 
            From Proposition~\ref{proposition-Identity} we have $Q_{k, \: l}^{*} \cong Q_{l-1, \: k+1}$,
            and thus, we also have
            \[\Rank_{\mathbb{K}} \: Q_{t, \: d+1-(s+t)}^{*} = \Rank_{\mathbb{K}} \: Q_{d-s-t, \: t+1}.\]
        \item
            For any $k \geq 0$, let us denote $e = d - (k+a)$.
            Observe that we have
            \begin{align*}
                \sum_{i=0}^a \Delta_i(k) &= \sum_{i=0}^{d+1 - k - (e+1)} \Delta_i(k) = \sum_{i=0}^{d+1 - k - (e+1)} \Rank_{\mathbb{K}} \: Q_{k, \: d+1-(i+k)} = \\
                &= \sum_{l=e+1}^{d+1 - k} \Rank_{\mathbb{K}} \: Q_{k, \: l}.
            \end{align*}
            Using Theorem~\ref{theorem-Iarrobino} (\ref{point-upper-margin}), we have
            \[\sum_{l=e+1}^{d+1 - k} \Rank_{\mathbb{K}} \: Q_{k, \: l} = \sum_{l=e+1}^{d+1} \Rank_{\mathbb{K}} \: Q_{k, \: l},\]
            so we need to show that it holds
            \[h_a(k) = \sum_{l=e+1}^{d+1} \Rank_{\mathbb{K}} \: Q_{k, \: l}.\]
            
            We have the filtrations
            \[A_{k, \: e} \subseteq A_{k, \: e+1} \subseteq A_{k, \: e+2} \subseteq  \ldots \subseteq A_{k, \: d} \subseteq A_{k, \: d+1} = \mathfrak{m}^{k}M,\]
            \[A_{k+1, \: e} + A_{k, \: e} \subseteq A_{k+1, \: e+1} + A_{k, \: e} \subseteq  \ldots \subseteq A_{k+1, \: d+1} + A_{k, \: e} = \mathfrak{m}^{k+1}M + A_{k, \: e}.\]
            
            We have that $A_{k+1, \: e+c} + A_{k, \: e} \subseteq A_{k, \: e+c}$ for $c = 0, \: 1, \ldots, d+1-e$.
            Also Remark~\ref{lemma-Modularity} gives us $A_{k, \: e+c} \cap (A_{k+1, \: e+c+1} + A_{k, \: e})
            = A_{k, \: e+c} \cap A_{k+1, \: e+c+1} + A_{k, \: e}$.
            We can then rewrite
            $A_{k, \: e+c} \cap A_{k+1, \: e+c+1} + A_{k, \: e} =
            \mathfrak{m}^{k}M \cap (0 : \mathfrak{m}^{e+c})_M \cap \mathfrak{m}^{k+1}M \cap (0 : \mathfrak{m}^{e+c+1})_M + A_{k, \: e} = \mathfrak{m}^{k+1}M \cap (0 : \mathfrak{m}^{e+c})_M + A_{k, \: e} = A_{k+1, \: e+c} + A_{k, \: e}$ for $c = 0, \: 1, \ldots, d-e$.
            Thus by Remark~\ref{lemma-filtration_quotient} we have the filtration
            \[0 = \frac{A_{k, \: e}}{A_{k+1, \: e} + A_{k, \: e}} \subseteq \frac{A_{k, \: e+1}}{A_{k+1, \: e+1} + A_{k, \: e}} \subseteq  \ldots \subseteq \frac{A_{k, \: d+1}}{A_{k+1, \: d+1} + A_{k, \: e}} = \frac{\mathfrak{m}^{k}M}{\mathfrak{m}^{k+1}M + A_{k, \: e}}.\]
            
            So after using Corollary~\ref{corollary-gradation_Hilbert} we have
            \begin{align*}
                h_a(k) &= \dim_{\mathbb{K}} \frac{\mathfrak{m}^{k}M}{\mathfrak{m}^{k+1}M + A_{k, \: e}} = \sum_{c = 0}^{d-e} \dim_{\mathbb{K}} \sfrac{ \left(\frac{A_{k, \: e+c+1}}{A_{k+1, \: e+c+1} + A_{k, \: e}} \right)}{ \left(\frac{A_{k, \: e+c}}{A_{k+1, \: e+c} + A_{k, \: e}} \right)} = \\
                &= \sum_{c = 0}^{d-e} \dim_{\mathbb{K}} \frac{A_{k, \: e+c+1}}{A_{k+1, \: e+c+1} + A_{k, \: e+c}} = \sum_{c = 0}^{d-e} \dim_{\mathbb{K}} Q_{k, \: e+c+1} = \sum_{l = e+1}^{d+1} \dim_{\mathbb{K}} Q_{k, \: l},
            \end{align*}
            which we wanted to obtain.
        \item
            We use Theorem~\ref{theorem-Iarrobino} (\ref{point-prefix}) with $a = d$. We only need to show that ${\bigoplus_{i=0}^a C_a^i} = 0$. We know that each $C_a^i$ is the image of $A_{i, \: d - (i+d)} = A_{i, \: -i} = \mathfrak{m}^iM \cap (0 : \mathfrak{m}^{-i})_M$. Since $-i$ is a non-positive integer, $\mathfrak{m}^{-i} = (1)$, so $(0 : \mathfrak{m}^{-i})_M = 0$. Thus $A_{i, \: -i} = 0$, which follows that each $C_a^i$ is zero.
    \end{enumerate}
\end{proof}

We show a handful of examples of applications of Iarrobino's symmetric decomposition in Section \ref{section-Classification}, but now we want to show a slight difference to the case of self-dual algebras.

\begin{example}
\label{example-Q_0_l}
    Iarrobino's symmetric decomposition for algebras also asserts that $Q_{0, \: l} = 0$ for $l < d+1$.
    This is not the case for modules.
    Let us consider the sequence (2, 1).
    We show in Lemma \ref{lemma-class-non-increasing} (see also the construction in the proof) that there exists a self-dual module with the local Hilbert function equal to this sequence with the decomposition (1, 1) $+$ (1). Thus, we have $1 = \Delta_1(0) = \Rank_{\mathbb{K}} \: Q_{0, \: 1}$,
    so $Q_{0, \: 1} \neq 0$ and $1 < d+1 = 2$.
\end{example}

\begin{remark}
\label{remark_D_surjection}
    From Theorem~\ref{theorem-Iarrobino}~(\ref{point-prefix}), after taking $a = 0$, we obtain a surjection from $\Gr M$ to $D_0$.
    Moreover, if the Iarrobino's symmetric decomposition is trivial (i.e., we have only one row of the decomposition), we have $\dim_{\mathbb{K}} \frac{m^k M}{m^{k+1}M} = h_0(k) = \dim_{\mathbb{K}} Q_{k, \: d-k}$ and the surjection becomes an isomorphism.
\end{remark}

\section{Apolarity}
\label{section-apolarity}

In this section, we recall the definition of an apolar module.
Then we prove the new result -- Kunte's criterion for self-duality.
The criterion was previously known in the graded case, we show a generalization to the non-graded case.
However, our result is only a sufficient statement, while in the graded case, the criterion was both sufficient and necessary.

Let $n$ be a positive natural number, and let us denote $S = \mathbb{K}[y_1, \: y_2, \: \ldots, \: y_n]$ and $S^{*} = \mathbb{K}[x_1, \: x_2, \: \ldots, \: x_n]$.
Let us also fix a positive natural number $r$.
We denote by $F^{*}$ the free $S^{*}$-module $\bigoplus_{k = 0}^r S^{*} e_k^{*}$ and by $F$ the free $S$-module $\bigoplus_{k = 0}^r S e_k$.
We introduce the $S$-module structure on $F^{*}$ in the following manner:
for $y_1^{a_1} y_2^{a_2} \ldots y_n^{a_n} \in S$ and $x_1^{b_1} x_2^{b_2} \ldots x_n^{b_n} \in S^{*}$ we define
\begin{equation*}
y_1^{a_1} y_2^{a_2} \ldots y_n^{a_n} \cdot x_1^{b_1} x_2^{b_2} \ldots x_n^{b_n} =
    \begin{cases}
        0 \text{, if there exists $k$ such that $b_k < a_k$} \\
        x_1^{b_1-a_1} x_2^{b_2-a_2} \ldots x_n^{b_n-a_n} \text{ in other cases}.
    \end{cases}
\end{equation*}
Note that in this definition, multiplication behaves as a derivation, where $y_k^{a_k}$ encodes differentiating $a_k$ times over $x_k$, and we normalize the multiplication constant. This action is usually called the contraction action.

Now for $f_k \in S^{*}$ we define
\[y_1^{a_1} y_2^{a_2} \ldots y_n^{a_n} \cdot \left( \sum_{k = 1}^r f_k e_k^{*} \right) = \sum_{k = 1}^r (y_1^{a_1} y_2^{a_2} \ldots y_n^{a_n} \cdot f_k) e_k^{*},\]
and we extend this definition to multiplication by the whole $S$ using linearity.

We also introduce a bilinear form $\langle \cdot , \cdot \rangle : F \times F^{*} \rightarrow \mathbb{K}$ by defining
\begin{equation*}
\langle y_1^{a_1} y_2^{a_2} \ldots y_n^{a_n} e_s, \: x_1^{b_1} x_2^{b_2} \ldots x_n^{b_n} e_l^{*} \rangle =
    \begin{cases}
        1 \text{, if $s = l$ and for all $k$ we have $a_k = b_k$} \\
        0 \text{ in other cases}
    \end{cases},
\end{equation*}
and extending it to the bilinear form by linearity.
We define $\perp$ with respect to this form, i.e. for an $S$-submodule $I \subseteq F$ we have that $I^{\perp}$ is an $S^*$-submodule of $F^{*}$ consisting of all elements $e$, such that $\langle I, \: e \rangle = 0$.
Analogously, for an $S^*$-submodule $J \subseteq F^{*}$ we have that $J^{\perp}$ is an $S$-submodule of $F$ consisting of all elements $e$, such that $\langle e, \: J \rangle = 0$.

\begin{definition}
\label{definition-apolarity}
For an element $f \in F^{*}$ we define the module apolar to $f$ by $\frac{F}{(S f)^{\perp}}$, where $\perp$ refers to the introduced bilinear form (i.e. $(S f)^{\perp}$ is an $S^*$-submodule consisting of all elements $e$ such that $\langle e , \: S f \rangle = 0$).

Similarly, for a tuple of elements $f_1, \: f_2, \ldots, f_l \in F^{*}$ we define the module apolar to $(f_1, \: f_2, \ldots, f_l)$ by $\frac{F}{(S f_1 + S f_2 + \ldots + S f_l)^{\perp}}$.
\end{definition}

\begin{theorem}
\label{theorem-every-apolar}
    Every finitely-generated module $M$ can be obtained as an apolar module.
\end{theorem}
\begin{proof}
    Since the module is finitely-generated, we can write $M = \frac{F^r}{K}$ for some module $K$. Let us take $K^\perp$. It is also finitely-generated, so we can take its generators $f_1, \: f_2, \ldots, f_s$. Now by Lemma~\ref{lemma-obvious} (\ref{lemma-double_perp}) we have $(K^\perp)^\perp \cong K$, so we have $M = \frac{F^r}{K} \cong \frac{F^r}{(Sf_1 + Sf_2 + \ldots + Sf_s)^\perp}$.
\end{proof}

We do not assume that the polynomials from $S$ and $S^{*}$ are graded.
Note that each $f_k \in F^{*}$ is encoded by the direct sum decomposition as $\sum_{i=0}^r f_{ki} e_i^*$, where $f_{ki} \in S^{*}$. Thus, the tuple $f_1, \: f_2, \ldots, f_l$ can be viewed as a matrix $N = [f_{ki}] \in \mathbb{M}_{l \times r}(S^{*})$.
We call this matrix the encoding matrix of the module apolar to $(f_1, \: f_2, \ldots, f_l)$.
Note also that we are interested in the case of self-dual modules, so we would like to have a tool that would allow us to determine if the obtained apolar module is self-dual.
It turns out that there is a very simple criterion for that in the case $l = r$, so if $N$ is square.
This criterion was introduced by Kunte in the graded case \cite[Theorem 1.1]{Kunte2008GorensteinMO}, we generalize it to the non-graded case.
However, our condition is only sufficient, while the one from the graded case was also necessary.

\begin{theorem-non}[Theorem~\ref{theorem-self-dual-square} -- Kunte's criterion]
\label{non-theorem-self-dual-square}
If the square matrix $N = [f_{ki}] \in \mathbb{M}_{r \times r}(S^{*})$ is symmetric, then the module apolar to $(f_k)_{k=1}^r$ is self-dual (in general non-graded), where $f_k \in F^*$.
\end{theorem-non}

Before proving the theorem, we show a more general result.
Let us recall that we have fixed a natural number $r$ and defined free modules $F = \bigoplus_{k = 0}^r S e_k$ and $F^* = \bigoplus_{k = 0}^r S^* e_k^*$.
Let us also fix a natural number $l$ and define $G = \bigoplus_{k = 0}^l S e_k$, $G^* = \bigoplus_{k = 0}^l S^* e_k^*$.
Let us also recall that we are considering a tuple $f_1, \: f_2, \ldots f_l \in F^*$ and we are using a decomposition $f_k = \sum_{i = 1}^r f_{ki} e_i^*$ and encoding this tuple as a matrix $N = [f_{ki}]$.
Let us now define a tuple $f_1^T, \: f_2^T, \ldots f_r^T \in G^*$ such that $f_i^T = \sum_{k = 1}^l f_{ki} e_k^*$.
Then the matrix encoding the tuple $f_1^T, \: f_2^T, \ldots f_r^T$ is $N^T$.
We show that the modules apolar to $(f_1, \: f_2, \ldots f_l)$ and $(f_1^T, \: f_2^T, \ldots f_r^T)$ are dual.
It shows that the dual module of an apolar module is obtained by transposing the encoding matrix.
In particular, if the matrix $N$ is symmetric, then this matrix is invariant under transposition, which implies that the module is self-dual.

\begin{theorem}
\label{theorem-transposition-isomorphism}
    The modules apolar to $(f_1, \: f_2, \ldots f_l)$ and $(f_1^T, \: f_2^T, \ldots f_r^T)$ are dual.
\end{theorem}

The result follows from two simpler lemmas.

\begin{lemma}
\label{lemma-apolarity-simplification}
    We have the $S$-module isomorphism $\frac{F}{(S f_1 + S f_2 + \ldots + S f_l)^{\perp}} \cong (S f_1 + S f_2 + \ldots + S f_l)^*$.
\end{lemma}
\begin{proof}
    Let $c > \max_i \deg f_i$.
    Then $S f_1 + S f_2 + \ldots + S f_l \subseteq F^*_{\leq c-1}$ and thus $(S f_1 + S f_2 + \ldots + S f_l)^{\perp} \supseteq F_{\geq c}$, so we have
    \[\frac{F}{(S f_1 + S f_2 + \ldots + S f_l)^{\perp}} \cong \frac{\frac{F}{F_{\geq c}}}{(S f_1 + S f_2 + \ldots + S f_l)^{\perp}}.\]
    Now let us observe that the pairing $F \times F^{*} \to \mathbb{K}$ induces the non-degenerate pairing $\frac{F}{F_{\geq c}} \times F^*_{\leq c-1} \to \mathbb{K}$, so we have an $S$-module isomorphism $\frac{F}{F_{\geq c}} \cong (F^*_{\leq c-1})^*$.
    Since also $S f_1 + S f_2 + \ldots + S f_l$ is a submodule of $F^*_{\leq c-1}$, we can use Lemma~\ref{lemma-Dual} for the module $F^*_{\leq c-1}$ and the submodule $S f_1 + S f_2 + \ldots + S f_l$, obtaining \[\frac{(F^*_{\leq c-1})^*}{(S f_1 + S f_2 + \ldots + S f_l)^{\perp}} \cong (S f_1 + S f_2 + \ldots + S f_l)^*.\]
    Since we have $(F^*_{\leq c-1})^* \cong \frac{F}{F_{\geq c}}$ we obtain the isomorphism \[\frac{\frac{F}{F_{\geq c}}}{(S f_1 + S f_2 + \ldots + S f_l)^{\perp}} \cong (S f_1 + S f_2 + \ldots + S f_l)^*.\]
    As we have already shown, it also holds $\frac{F}{(S f_1 + S f_2 + \ldots + S f_l)^{\perp}} \cong \frac{\frac{F}{F_{\geq c}}}{(S f_1 + S f_2 + \ldots + S f_l)^{\perp}}$, which implies the claim.
\end{proof}

\begin{lemma}
\label{lemma-transposition-duality}
    We have the $S$-module isomorphism $(S f_1 + S f_2 + \ldots + S f_l) \cong (S f_1^T + S f_2^T + \ldots + S f_r^T)^*$.
\end{lemma}
\begin{proof}
    Let $f$ be a column vector of $f_i$ and let $f_T$ be a column vector of $(f_i^T)$.
    By Lemma~\ref{lemma-perfect-pairing} it is sufficient to show that there exists a non-degenerate pairing
    \[\langle \: , \:  \rangle: (S f_1 + S f_2 + \ldots + S f_l) \times (S f_1^T + S f_2^T + \ldots + S f_r^T) \rightarrow \mathbb{K},\]
    such that for every
    $s \in S, \: \sigma_1 \in S^l, \: \sigma_2 \in S^l$ we have
    $\langle s \cdot \sigma_1^T f, \: \sigma_2^T f_T \rangle =
    \langle \sigma_1^T f, \: s \cdot \sigma_2^T f_T \rangle$.
    
    Let $N$ be a matrix encoding $f_1, \: f_2, \ldots f_l$, i.e. $N = [f_{ki}]$.
    Let $\pi: S^* \rightarrow \mathbb{K}$ be a functional that sends $1$ in $S^*$ to $1$ in $\mathbb{K}$.
    For elements $\sigma_1 = (\sigma_{11}, \: \sigma_{12}, \ldots, \sigma_{1l}) \in S^l$ and $\sigma_2 = (\sigma_{21}, \: \sigma_{22}, \ldots, \sigma_{2r}) \in S^r$,
    we consider the pairing
    \[\langle \sigma_{11} f_1 + \sigma_{12} f_2 + \ldots \sigma_{1l} f_l, \: \sigma_{21} f_1^T + \sigma_{22} f_2^T + \ldots \sigma_{2r} f_r^T \rangle = \pi \left(\sum_{i, \: k} \sigma_{1i} \sigma_{2k} \: f_{ik} \right),\]
    which can be written in a more compact form as
    \[\langle \sigma_1^T f, \: \sigma_2^T f_T \rangle = \pi(\sigma_1^T N \sigma_2).\]
    Then for $s \in S$ we have $\langle s \cdot \sigma_1^T f, \: \sigma_2^T f_T \rangle = \langle (s \cdot \sigma_1^T) f, \: \sigma_2^T f_T \rangle = \pi(s \cdot \sigma_1^T N \sigma_2)$.
    Similarly, $\langle \sigma_1^T f, \: s \cdot \sigma_2^T f_T \rangle = \langle \sigma_1^T f, \: (s \cdot \sigma_2^T) f_T \rangle = \pi(\sigma_1^T N \cdot s \cdot \sigma_2) = \pi(s \cdot \sigma_1^T N \sigma_2)$, where in the last equality we used commutativity of $S$.
    We show now that this pairing is also non-degenerate.
    
    Let us observe that the rows of the matrix $N$ are vectors $f_k$ and its columns are vectors $f_i^T$.
    Thus we have $\sigma_1^T N = \sigma_1^T f$ and $N \sigma_2 = f_T^T \sigma_2 = (\sigma_2^T f_T)^T$.
    So if it holds that $\sigma_1^T f \neq 0$, then it also holds that $\sigma_1^T N \neq 0$.
    Thus there exists such an element $\sigma_2$ that $(\sigma_1^T N) \sigma_2 = 1$: we can take such an index $j$ that
    $(\sigma_1^T N)_j \neq 0$, then there exists an element $a \in S$ such that
    $a \cdot (\sigma_1^T N)_j = 1$ and we can take $\sigma_{2j} = a$ and $\sigma_{2s} = 0$ for $s \neq j$.
    Then $\pi(\sigma_1^T N \sigma_2) = \pi(1) = 1$.
    Similarly, if it holds that $\sigma_2^T f \neq 0$, then it also holds that $N \sigma_2 \neq 0$.
    Thus there exists such an element $\sigma_1$ that $\sigma_1^T (N \sigma_2) = 1$ (analogical argument as above).
    Then $\pi(\sigma_1^T N \sigma_2) = \pi(1) = 1$.
    So the pairing is indeed non-degenerate.
\end{proof}

\begin{proof}[Proof of Theorem~\ref{theorem-transposition-isomorphism}]
    We want to show duality between $\frac{F}{(S f_1 + S f_2 + \ldots + S f_l)^{\perp}}$ and $\frac{G}{(S f_1^T + S f_2^T + \ldots + S f_r^T)^{\perp}}$.
    Using Lemma~\ref{lemma-apolarity-simplification} for the tuple $(f_1, \: f_2, \ldots f_l)$ and the module $F$ and for the tuple $(f_1^T, \: f_2^T, \ldots f_r^T)$ and the module $G$ we obtain $\frac{F}{(S f_1 + S f_2 + \ldots + S f_l)^{\perp}} \cong (S f_1 + S f_2 + \ldots + S f_l)^*$ and $\frac{G}{(S f_1^T + S f_2^T + \ldots + S f_r^T)^{\perp}} \cong (S f_1^T + S f_2^T + \ldots + S f_r^T)^*$.
    So we want to show $(S f_1 + S f_2 + \ldots + S f_l)^{**} \cong (S f_1^T + S f_2^T + \ldots + S f_r^T)^*$.
    By Lemma~\ref{lemma-bidual-module} we have $(S f_1 + S f_2 + \ldots + S f_l)^{**} \cong S f_1 + S f_2 + \ldots + S f_l$, so we need to show $(S f_1 + S f_2 + \ldots + S f_l) \cong (S f_1^T + S f_2^T + \ldots + S f_r^T)^*$, which is true thanks to Lemma~\ref{lemma-transposition-duality}.
\end{proof}

\begin{proof}[Proof of Theorem~\ref{theorem-self-dual-square}]
\label{proof-theorem-self-dual-square}
    If $N$ is symmetric, then $N$ is invariant under transposition, which by Theorem~\ref{theorem-transposition-isomorphism} implies that the module apolar to $(f_k)_{k=1}^r$ is self-dual.
\end{proof}

Let us now introduce a tool that simplifies the computation of the local Hilbert function of an apolar module.
Formally, for an apolar module $M$ we consider it as a module over a ring $S_M = \frac{S}{Ann(M)}$, which is a local ring with the unique maximal ideal $S_{\geq 1} \cdot S_M$.
Clearly, if $M$ is self-dual as an $S$-module, then it is also self-dual as an $S_M$-module.
By definition we have $h_M(k) = \dim_{\mathbb{K}} \frac{S_{\geq k} M}{S_{\geq k+1} M}$.
However, computing the local Hilbert function from this formula is usually inconvenient, and thus, we propose another approach.
First, we need to introduce a definition.

\begin{definition}
\label{definition-degree}
    For an element $f = \sum_{i=1}^r f_i e_i^* \in F$ we define a degree of $f$ by taking
    $\deg f = \max_{i=1, \: 2, \ldots, \: r} \deg f_i$.
\end{definition}

\begin{remark}
\label{remark-degree-maximal}
    For every non-negative integer $c$ and every element $f \in F^*$ we have the following equivalence:
    $\deg f < c \iff S_{\geq c} f = 0$.
\end{remark}

\begin{proposition}
\label{proposition-Hilbert-apolar-computation}
    The local Hilbert function $h_M$ of an apolar finite-dimensional self-dual module $M$ satisfies the following formula:
    \[h_M(k) = \dim_{\mathbb{K}} \{f \in M: \deg f \leq k\} - \dim_{\mathbb{K}} \{f \in M: \deg f \leq k-1 \}.\]
\end{proposition}
\begin{proof}
    Remark~\ref{remark-degree-maximal} implies that we have the equality $\{f \in M: \deg f \leq k\} = (0 : S_{\geq k+1})_M$.
    From Corollary~\ref{corollary-Annihilator_m} we obtain that $(0 : S_{\geq k+1})_M \cong (S_{\geq k+1})_M^{\perp}$, where $\perp$ refers to Definition \ref{definition-perp}.
    Now we can rewrite using Lemma~\ref{lemma-obvious} (\ref{lemma-dim_perp}):
    $\dim_{\mathbb{K}} \{f \in M: \deg f \leq k\} - \dim_{\mathbb{K}} \{f \in M: \deg f \leq k-1 \} =
    \dim_{\mathbb{K}} (S_{\geq k+1})_M^{\perp} - \dim_{\mathbb{K}} (S_{\geq k})_M^{\perp} =
    (\dim_{\mathbb{K}} M - \dim_{\mathbb{K}} (S_{\geq k+1})_M) - 
    (\dim_{\mathbb{K}} M - \dim_{\mathbb{K}} (S_{\geq k})_M) = 
    \dim_{\mathbb{K}} (S_{\geq k})_M - \dim_{\mathbb{K}} (S_{\geq k+1})_M = 
    \dim_{\mathbb{K}} \frac{S_{\geq k} M}{S_{\geq k+1} M} = h_M(k)$.
\end{proof}

Now we show two examples of applications of Proposition~\ref{proposition-Hilbert-apolar-computation} and Theorem~\ref{theorem-self-dual-square} to computations of the local Hilbert functions of self-dual modules.
The first example is the case where $M$ is an algebra.
It is a special case of Theorem~\ref{theorem-self-dual-square}, when the encoding matrix is $1 \times 1$, so it is trivially symmetric.
The second example is for the $2 \times 2$ encoding matrix.

\begin{example}
\label{example-apolar-algebra}
    In this example, we show that the module (algebra) $M$ apolar to $x^3 + y^3 + z^2$ is self-dual and compute its local Hilbert function.
    The fact that the module $M$ is self-dual follows directly from Theorem~\ref{theorem-self-dual-square}, as explained above (the encoding matrix is $1 \times 1$, so it is trivially symmetric).
    Now, let us compute the local Hilbert function using the formula from Proposition~\ref{proposition-Hilbert-apolar-computation}.
    
    Let us denote $f = x^3 + y^3 + z^2$, $Z_k = \{g \in M: \deg g \leq k\}$.
    Let $W_k$ be some $\mathbb{K}$-linear space satisfying $Z_k = W_k \oplus Z_{k-1}$, where we take $Z_{-1} = 0$ (we are interested in the dimension of $W_k$ over $\mathbb{K}$, so we just want to find some space $W_k$ satisfying this condition).
    Then we can take as $W_3, \: W_2, \: W_1, \: W_0$ the $\mathbb{K}$-linear subspaces spanned by respectively $(x^3 + y^3 + z^2)$, $(x^2, \: y^2)$, $(x, \: y, \: z)$, 
    $(1)$.
    Now we have $h_M(k) = \dim_{\mathbb{K}} Z_k - \dim_{\mathbb{K}} Z_{k-1} = \dim_{\mathbb{K}} W_k$, so the local Hilbert function is (1, 3, 2, 1).
\end{example}

\begin{example}
\label{example-apolar-module}
    In this example, we show that the module apolar to the tuple $((x^2 + y^2) e_1^* + (x^3 + y^2) e_2^*, \: (x^3 + y^2) e_1^* + (x^5 + y^4) e_2^*)$ is self-dual and compute its local Hilbert function. This method can be utilized to compute the local Hilbert function in Example~\ref{example-kunte-diagonal}
    The fact that the module is self-dual follows from Theorem~\ref{theorem-self-dual-square}, as its encoding matrix \[N = \begin{bmatrix}
    x^2 + y^2 & x^3 + y^2 \\
    x^3 + y^2 & x^5 + y^4 \\
    \end{bmatrix}\]
    is symmetric.
    Now we compute the local Hilbert function using the formula from Proposition~\ref{proposition-Hilbert-apolar-computation}.
    We follow a very similar approach to that of Example~\ref{example-apolar-algebra}.
    
    Let us denote $f_1 = (x^2 + y^2) e_1^* + (x^3 + y^2) e_2^*$, $f_2 = (x^3 + y^2) e_1^* + (x^5 + y^4) e_2^*$, $Z_k = \{g \in M: \deg g \leq k\}$.
    Let $D_k$ be a $\mathbb{K}$-linear subspace such that $Z_k = Z_{k-1} + D_k$, where we take $Z_{-1} = 0$.
    Observe that we now have $\dim_{\mathbb{K}} Z_k - \dim_{\mathbb{K}} Z_{k-1} = \dim_{\mathbb{K}} D_k - \dim_{\mathbb{K}} D_k \cap Z_{k-1}$.
    Let $E_k = D_k \cap Z_{k-1}$.
    We can choose the pairs of $\mathbb{K}$-linear subspaces $(D_k, \: E_k)$ for $k = 5, \: 4, \: 3, \: 2, \: 1, \: 0$ as follows:
    \begin{enumerate}
        \item[k = 5:]
            $D_k = \Span \{(x^3 + y^2) e_1^* + (x^5 + y^4) e_2^*\}$, $E_k = 0$;
        \item[k = 4:]
            $D_k = \Span \{x^2 e_1^* + x^4 e_2^*\}$, $E_k = 0$;
        \item[k = 3:]
            $D_k = \Span \{x e_1^* + x^3 e_2^*, \: y e_1^* + y^3 e_2^*, \: (x^2 + y^2) e_1^* + (x^3 + y^2) e_2^*\}$, $E_k = \Span \{(x^2 + y^2 - x) e_1^* + y^2 e_2^*\}$;
            
        \item[k = 2:]
            $D_k = \Span \{e_1^* + x^2 e_2^*, \: e_1^* + y^2 e_2^*, \: x e_1^* + x^2 e_2^*, \: (x^2 + y^2 - x) e_1^* + y^2 e_2^*\}$, $E_k = \Span \{(x-1) e_1^*\}$;
        \item[k = 1:]
            $D_k = \Span \{x e_2^*, \: y e_2^*, \: y e_1^* + y e_2^*, \: e_1^* + x e_2^*, \: (x-1) e_1^*\}$, $E_k = \Span \{e_1^*\}$;
        \item[k = 0:]
            $D_k = \Span \{e_1^*, \: e_2^*\}$, $E_k = 0$;
    \end{enumerate}
    We have $h_M(k) = \dim_{\mathbb{K}} Z_k - \dim_{\mathbb{K}} Z_{k-1} = \dim_{\mathbb{K}} D_k - \dim_{\mathbb{K}} E_k$, so the local Hilbert function is (2, 4, 3, 2, 1, 1).
\end{example}

\section{Classification for small degrees}
\label{section-Classification}

In this section, we classify all possible local Hilbert functions for self-dual modules of small degree (so we assume here that $R$ is local and $M$ is self-dual) using tools we have introduced before -- Iarrobino's symmetric decomposition, Macaulay's Bound, and the apolarity.
To the best of our knowledge, this is the first such classification.
B. Poonen classified the local Hilbert functions and the types of isomorphisms, but only for algebras of a degree less than or equal to six \cite{Poonen}.
First, we introduce a few useful lemmas that simplify the task of classification (some of which do not assume that $M$ is self-dual -- we note this then in the statements).

We denote as $h_{M_1} + h_{M_2}$ the component-wise addition of the local Hilbert functions.

\begin{proposition}
\label{proposition-direct-sum}
    Let $M_1, \: M_2$ be self-dual modules and $h_{M_1}, \: h_{M_2}$ be their Hilbert functions.
    Then the module $M_1 \oplus M_2$ is a self-dual module with the local Hilbert function $h_{M_1} + h_{M_2}$.
\end{proposition}
\begin{proof}
    We have $(M_1 \oplus M_2)^* \cong M_1^* \oplus M_2^* \cong M_1 \oplus M_2$, so the module $M_1 \oplus M_2$ is indeed self-dual.
    Then the unique maximal ideals in $M_1$, $M_2$, $M_1 \oplus M_2$ are, respectively, $\mathfrak{m} M_1$, $\mathfrak{m} M_2$, $\mathfrak{m} (M_1 \oplus M_2)$.
    Now we have $h_{M_1 \oplus M_2}(k) = \dim_{\mathbb{K}} \frac{\mathfrak{m}^k (M_1 \oplus M_2)}{\mathfrak{m}^{k+1} (M_1 \oplus M_2)} = \dim_{\mathbb{K}} (\frac{\mathfrak{m}^k M_1}{\mathfrak{m}^{k+1} M_1} \oplus \frac{\mathfrak{m}^k M_2}{\mathfrak{m}^{k+1} M_2}) = \dim_{\mathbb{K}} (\frac{\mathfrak{m}^k M_1}{\mathfrak{m}^{k+1} M_1}) + \dim_{\mathbb{K}} (\frac{\mathfrak{m}^k M_2}{\mathfrak{m}^{k+1} M_2}) = h_{M_1}(k) + h_{M_2}(k)$.
\end{proof}
\begin{corollary}
\label{corollary-long-direct-sum}
    Let $M_1, \: M_2, \ldots, M_s$ be self-dual modules and $h_{M_1}, \: h_{M_2} \ldots, h_{M_s}$ be their local Hilbert functions.
    Then the module $M_1 \oplus M_2 \oplus \ldots \oplus M_s$ is a self-dual module with the local Hilbert function $h_{M_1} + h_{M_2} + \ldots + h_{M_s}$.
\end{corollary}
\begin{proof}
    It follows directly from Proposition~\ref{proposition-direct-sum} by induction.
\end{proof}

\begin{lemma}
\label{lemma-class-all-ones}
    The sequence $(1, 1, \ldots, 1)$ is the local Hilbert function of some self-dual module for every positive length of the sequence.
\end{lemma}
\begin{proof}
    Let us denote the length of the sequence of ones by $m$.
    Then $\frac{\mathbb{K}[x]}{x^m}$ is a self-dual module with the local Hilbert function $(1, 1, \ldots, 1)$.
\end{proof}

\begin{lemma}
\label{lemma-class-zeros-tail}
    If $h_M$ is the local Hilbert function of a module $M$ (not necessarily self-dual) and $h_M(k) = 0$ for some $k$, then $h_M(l) = 0$ for $l > k$.
\end{lemma}
\begin{proof}
    We prove by induction that for $c \geq 0$ we have $h_M(k+c) = 0$.
    For $c = 0$ it is true thanks to the assumption.
    Now if $h_M(k+c) = 0$, then we have that $h_M(k+c) \leq k+c$ and thus, thanks to Corollary~\ref{corollary-Macaulay_d}, we have $h_M(k+c+1) \leq h_M(k+c) = 0$.
\end{proof}

\begin{lemma}
\label{lemma-class-ones-tail}
    If $h_M$ is the local Hilbert function of a module $M$ (not necessarily self-dual) and $h_M(k) = 1$ for some $k \geq 1$, then $h_M(l) \leq 1$ for $l > k$.
\end{lemma}
\begin{proof}
    We prove by induction that for $c \geq 0$ we have $h_M(k+c) \leq 1$.
    For $c = 0$ it is true thanks to the assumption.
    Now if $h_M(k+c) \leq 1$, then since $k \geq 1$ we have that $h_M(k+c) \leq k+c$ and thus, thanks to Corollary~\ref{corollary-Macaulay_d}, we have $h_M(k+c+1) \leq h_M(k+c) \leq 1$.
\end{proof}

\begin{lemma}
\label{lemma-class-non-increasing}
    If a sequence $l$ is non-increasing, then it is the local Hilbert function of some self-dual module.
\end{lemma}
\begin{proof}
    We prove it by induction on the sum of elements in the sequence.
    If the sequence is $(1, 1, \ldots, 1)$, then it is the local Hilbert function of some self-dual module, thanks to Lemma~\ref{lemma-class-all-ones}.
    In particular, it holds for $(1)$, which is the base of the induction.
    
    Now observe that if the sequence $l$ is non-increasing, then it can be decomposed into a sum $(1, 1, \ldots, 1) + l^\prime$, where $l^\prime$ is non-increasing.
    If $l \neq (1, 1, \ldots, 1)$, then $l^\prime$ is non-empty and thus, by the induction hypothesis, it is the local Hilbert function of some self-dual module.
    Now observe that since $(1, 1, \ldots, 1)$ and $l^\prime$ are both local Hilbert functions of self-dual modules, then by Proposition~\ref{proposition-direct-sum} their sum is also the local Hilbert function of some self-dual module.
\end{proof}

\begin{lemma}
\label{lemma-class-first-last}
    If $h_M$ is the local Hilbert function of a self-dual module $M$ and $k$ is the largest natural number such that $h_M(k) > 0$, then $h_M(0) \geq h_M(k)$.
\end{lemma}
\begin{proof}
    Let us consider Iarrobino's symmetric decomposition of $h_M$ -- the first row must be symmetric from Theorem~\ref{theorem-Iarrobino} (\ref{point-prefix}), and must have $h_M(k)$ on the last position (since next rows are strictly shorter), so it also must have $h_M(k)$ on the first position.
    Since $h_M(0)$ is a sum of values on the first positions in all rows, we have $h_M(0) \geq h_M(k)$.
\end{proof}

\begin{lemma}
\label{lemma-class-symmetry-center}
    Let $h_M$ be the local Hilbert function of a self-dual module $M$ and let $k$ be the largest natural number such that $h_M(k) > 0$. Then $\sum_{i = 0}^{\lfloor{\frac{k-1}{2}} \rfloor } h_M(i) \geq \sum_{i = \lceil{\frac{k+1}{2}} \rceil}^{k} h_M(i)$.
\end{lemma}
\begin{proof}
    The claim states that the sum of the first half of the sequence is greater than or equal to the sum of the second half.
    This follows from Iarrobino's symmetric decomposition (Theorem~\ref{theorem-Iarrobino}) -- the sequence is a sum of symmetric parts, and each of these parts contributes to the first half at least as much as it does to the second half.
\end{proof}

We say that the local Hilbert function is symmetric, if $h_M(i) = h_M(k-i)$, where $k$ is the largest natural number such that $h_M(k) > 0$.
For a sequence $s = (s_0, \: s_1, \ldots, \: s_c)$ and $0 \leq a \leq b \leq c$ we denote by $s_{a:b}$ the subsequence $(s_a, \: s_{a+1}, \ldots, \: s_b)$

\begin{lemma}
\label{lemma-symmetric-function}
    Let $h_M$ be the local Hilbert function of a self-dual module $M$, and let us assume that $h_M$ is symmetric. Then, Iarrobino's symmetric decomposition of $h_M$ has only one row.
\end{lemma}
\begin{proof}
    We prove the claim for decompositions of symmetric sequences (so in the proof, we assume only that $h_M$ is a symmetric sequence).
    We prove this by induction on the length $s$ of $h_M$.
    For $s = 1$ and $s = 2$ the claim is clear.
    Let the first row of Iarrobino's symmetric decomposition be $l_{0:k}$. 
    Then we have $l_0 = l_k = h_M(k) = h_M(0)$.
    Let us consider the subsequence $(h_M)_{1:k-1}$.
    This sequence is symmetric, so its decomposition has only one row.
    On the other hand, the decomposition of $h_M$ induces the decomposition of ${h_M}_{1:k-1}$.
    So positions from $1$ to $k-1$ are fulfilled by only one row.
    Since $l_0 = l_k = h_M(k) = h_M(0)$, positions $0$ and $k$ are fulfilled by the first row.
    Thus, the whole decomposition has only one row.
\end{proof}

\begin{lemma}
\label{lemma-class-non-increasing-algebra}
    Let $l = (l_0, \: l_1, \ldots, l_t)$ be a sequence such that $l_0 = 1$, $l_{1:t}$ is non-increasing, and $l_t = 1$.
    Then the sequence $l$ is the local Hilbert function of some self-dual module (algebra).
\end{lemma}
\begin{proof}
    We construct a polynomial $w$ such that $l$ is the local Hilbert function of the algebra apolar to $w$.
    Let us denote the polynomial $w$ as a sum of monomials $w = \sum_{i, \: j \in Z} a_{i j} x_i^j$, where $Z$ is a set that we will determine later.
    We take $a_{i j} = 1$ for all $i, \: j \in Z$.
    Now we determine the set $Z$ consisting of non-zero coefficients of $w$.
    Let us denote $c_n = l_{n} - l_{n+1}$ (we set $l_{t+1} = 0$).
    Since $l_{1:t}$ is non-increasing, we have that $l_1 = c_1 + c_2 + \ldots + c_t$.
    Now we take $l_1$ variables $x_1, \: x_2, \ldots, x_{l_1}$.
    For each variable, we take exactly one power for which the coefficient is non-zero (i.e., for each $i$ there exists a unique $j$ such that ${i, \: j \in Z}$), so let us denote $j(i)$.
    We set $j(1) = t$.
    Let $s$ be a variable that is equal to the greatest $i$ for which $j(i)$ has already been set (so at the beginning we have $s = 1$).
    Let us observe that we have $c_t = 1$ and from the assumption that $l_{1:t}$ is non-increasing, we have that $c_1, \: c_2, \ldots, c_t$ are non-negative.
    Now we iterate over the sequence $c_{t-1}, c_{t-2}, \ldots, c_1$ and if $c_n$ is non-zero, then we set $j(s+1), \: j(s+2), \ldots, j(s+c_n) = n+1$ and we increase the variable $s$ by assigning $s := s+c_n$.
    Thanks to Proposition~\ref{proposition-Hilbert-apolar-computation}, we obtain that the above algorithm creates an apolar algebra $A$ with the local Hilbert function $h_A$ satisfying $h_A(n) = l_n$ for $n \geq 1$.
    Clearly, $h_A(0) = 1 = l_0$, so we have constructed the desired self-dual algebra.
\end{proof}

In this example, we show how the construction from Lemma~\ref{lemma-class-non-increasing-algebra} works in practice.
\begin{example}
    Let us choose $l = (1, 5, 3, 3, 1)$. Then we have $t = 4$, $l_{1:t} = (5, 3, 3, 1)$, $c_1 = 2$, $c_2 = 0$, $c_3 = 2$, $c_4 = 1$.
    We search for the polynomial $w$ in $l_1 = 5$ variables, so $w = x_1^{j(1)} + x_2^{j(2)} + x_3^{j(3)} + x_4^{j(4)} + x_5^{j(5)}$.
    First, we set $j(1) = t = 4$.
    Then, we iterate over $c_3, \: c_2, \: c_1$.
    We begin with $n = 3$, and we have $c_3 = 2$, so we set $j(2) = j(3) = 4$.
    Then we set $s := 4$.
    In the next step, we have $n = 2$, and $c_2 = 0$, so we proceed to $n = 1$.
    In the last step, we have $n = 1$, and $c_1 = 2$, so we set $j(4) = j(5) = 2$.
    Thus, we obtain $w = x_1^4 + x_2^4 + x_3^4 + x_4^2 + x_5^2$.
    After applying Proposition~\ref{proposition-Hilbert-apolar-computation}, we obtain that the local Hilbert function of the algebra apolar to $w$ is precisely $l$.
    
\end{example}

\begin{lemma}
\label{lemma-graded-self-dual}
    Let us assume that there exists a self-dual module with the local Hilbert function that has a trivial Iarrobino's symmetric decomposition (having one row).
    Then there exists a graded apolar self-dual module with the same local Hilbert function.
\end{lemma}
\begin{proof}
    From Remark~\ref{remark_D_surjection} we have the $\Gr R$-modules isomorphism $\Gr M \cong D_0$.
    From Proposition~\ref{proposition_D_self-dual} $D_0$ is a graded self-dual module.
    From Theorem~\ref{theorem-every-apolar} we can also assume that it is apolar.
\end{proof}

Now, we classify possible local Hilbert functions of self-dual modules given the (small) degree.
We omit the sequences that trivially are not local Hilbert functions of self-dual modules because of the previous lemmas.
We denote by $m$ the degree of the considered modules.

\paragraph{m = 1:}
    By Lemma~\ref{lemma-class-non-increasing} possible is:
    \begin{align}
        (1).
    \end{align}

\paragraph{m = 2:}
    By Lemma~\ref{lemma-class-non-increasing} possible are:
    \begin{align}
        (1, 1), (2).
    \end{align}

\paragraph{m = 3:}
    By Lemma~\ref{lemma-class-non-increasing} possible are:
    \begin{align}
        (1, 1, 1), (2, 1), (3).
    \end{align}

\paragraph{m = 4:}
    By Lemma~\ref{lemma-class-non-increasing} possible are:
    \begin{align}
        (1, 1, 1, 1), (2, 1, 1), (2, 2), (3, 1), (4).
    \end{align}
    \begin{enumerate}
        \item[(1, 2, 1)]:
            possible -- module (algebra) apolar to $x^2 + y^2$.
    \end{enumerate}

\paragraph{m = 5:}
    By Lemma~\ref{lemma-class-non-increasing} possible are:
    \begin{align}
        (1, 1, 1, 1, 1), (2, 1, 1, 1), (2, 2, 1), (3, 1, 1), (3, 2), (4, 1), (5).
    \end{align}
    \begin{enumerate}
        \item[(1, 2, 1, 1)]:
            possible -- module (algebra) apolar to $x^3 + y^2$,
        \item[(1, 3, 1)]:
            possible -- module (algebra) apolar to $x^2 + y^2 + z^2$.
    \end{enumerate}

\paragraph{m = 6:}
    By Lemma~\ref{lemma-class-non-increasing} possible are:
    \begin{align}
        (1, 1, 1, 1, 1, 1), (2, 1, 1, 1, 1), (2, 2, 1, 1), (2, 2, 2), \\
        (3, 1, 1, 1), (3, 2, 1), (3, 3), (4, 1, 1), (4, 2), (5, 1), (6).
    \end{align}
    \begin{enumerate}
        \item[(1, 2, 1, 1, 1)]:
            possible -- module (algebra) apolar to $x^4 + y^2$,
        \item[(1, 2, 2, 1)]:
            possible -- module (algebra) apolar to $x^3 + y^3$,
        \item[(1, 3, 1, 1)]:
            possible -- module (algebra) apolar to $x^3 + y^2 + z^2$,
        \item[(1, 4, 1)]:
            possible -- module (algebra) apolar to $x^2 + y^2 + z^2 + w^2$,
        \item[(2, 3, 1)]:
            possible -- direct sum (1, 1) $+$ (1, 2, 1).
    \end{enumerate}
    
\paragraph{m = 7:}
    By Lemma~\ref{lemma-class-non-increasing} possible are:
    \begin{align}
        (1, 1, 1, 1, 1, 1, 1), (2, 1, 1, 1, 1, 1), (2, 2, 1, 1, 1), (2, 2, 2, 1), \\
        (3, 1, 1, 1, 1), (3, 2, 1, 1), (3, 2, 2), (3, 3, 1), \\
        (4, 1, 1, 1), (4, 2, 1), (4, 3), (5, 1, 1), (5, 2), (6, 1), (7).
    \end{align}
    \begin{enumerate}
        \item[(1, 2, 1, 1, 1, 1)]:
            possible -- module (algebra) apolar to $x^5 + y^2$,
        \item[(1, 2, 2, 1, 1)]:
            possible -- module (algebra) apolar to $x^4 + y^3$,
        \item[(1, 3, 1, 1, 1)]:
            possible -- module (algebra) apolar to $x^4 + y^2 + z^2$,
        \item[(1, 3, 2, 1)]:
            possible -- module (algebra) apolar to $x^3 + y^3 + z^2$,
        \item[(1, 4, 1, 1)]:
            possible -- module (algebra) apolar to $x^3 + y^2 + z^2 + w^2$,
        \item[(1, 5, 1)]:
            possible -- module (algebra) apolar to $x^2 + y^2 + z^2 + w^2 + u^2$,
        \item[(2, 3, 1, 1)]:
            possible -- direct sum (1, 1) $+$ (1, 2, 1, 1),
        \item[(2, 3, 2)]:
            possible -- direct sum (1, 1, 1) $+$ (1, 2, 1),
        \item[(2, 4, 1)]:
            possible -- direct sum (1, 1) $+$ (1, 3, 1).
    \end{enumerate}

\paragraph{m = 8:}
    By Lemma~\ref{lemma-class-non-increasing} possible are:
    \begin{align}
         (1, 1, 1, 1, 1, 1, 1, 1), (2, 1, 1, 1, 1, 1, 1), (2, 2, 1, 1, 1, 1), (2, 2, 2, 1, 1), (2, 2, 2, 2), \\
         (3, 1, 1, 1, 1, 1), (3, 2, 1, 1, 1), (3, 2, 2, 1), (3, 3, 1, 1), (3, 3, 2)\\
         (4, 1, 1, 1, 1), (4, 2, 1, 1), (4, 2, 2), (4, 3, 1), (4, 4) \\
         (5, 1, 1, 1), (5, 2, 1), (5, 3), (6, 1, 1), (6, 2), (7, 1), (8).
    \end{align}

    \begin{enumerate}
        \item[(1, 2, 1, 1, 1, 1, 1)]:
            possible -- module (algebra) apolar to $x^6 + y^2$,
        \item[(1, 2, 2, 1, 1, 1)]:
            possible -- module (algebra) apolar to $x^5 + y^3$,
        \item[(1, 2, 2, 2, 1)]:
            possible -- module (algebra) apolar to $x^4 + y^4$,
        \item[(1, 2, 3, 1, 1)]:
            impossible due to Iarrobino's symmetric decomposition -- let us assume that there exists a decomposition; then the first row must be (1, a, b, a, 1) (due to the last position) and so the second is (0, c, c, 0) and the third is (0, d, 0) (due to the first position); now analyzing the fourth position gives us $a = 1$; since from Theorem~\ref{theorem-Iarrobino} (\ref{point-prefix}) the first row is the local Hilbert function of some module and thus thanks to Lemma~\ref{lemma-class-ones-tail} we obtain $b = 1$; so analyzing the second position gives us $2 = 1 + c + d$ and analyzing the third position gives us $3 = 1 + c$, which implies $d < 0$ -- contradiction,
        \item[(1, 3, 1, 1, 1, 1)]:
            possible -- module (algebra) apolar to $x^5 + y^2 + z^2$,
        \item[(1, 3, 2, 1, 1)]:
            possible -- module (algebra) apolar to $x^4 + y^3 + z^2$,
        \item[(1, 3, 3, 1)]:
            possible -- module (algebra) apolar to $x^3 + y^3 + z^3$,
        \item[(1, 4, 1, 1, 1)]:
            possible -- module (algebra) apolar to $x^4 + y^2 + z^2 + w^2$,
        \item[(1, 4, 2, 1)]:
            possible -- module (algebra) apolar to $x^3 + y^3 + z^2 + w^2$,
        \item[(1, 5, 1, 1)]:
            possible -- module (algebra) apolar to $x^3 + y^2 + z^2 + w^2 + u^2$,
        \item[(1, 6, 1)]:
            possible -- module (algebra) apolar to $x^2 + y^2 + z^2 + w^2 + u^2 + v^2$,
        
        \item[(2, 2, 3, 1)]:
            possible -- module apolar to $(x^3+y^3)e_1^* + xye_2^*, xye^*_1 + ye_2^*$, this module is self-dual from Theorem~\ref{theorem-self-dual-square} with encoding matrix
            \[\begin{bmatrix}
                x^3 + y^3 & xy \\
                xy & y \\
            \end{bmatrix},\]
        \item[(2, 3, 1, 1, 1)]:
            possible -- direct sum (1, 1) $+$ (1, 2, 1, 1, 1),
        \item[(2, 3, 2, 1)]:
            possible -- direct sum (1, 1, 1, 1) $+$ (1, 2, 1),
        \item[(2, 4, 1, 1)]:
            possible -- direct sum (1, 1) $+$ (1, 3, 1, 1),
        \item[(2, 4, 2)]:
            possible -- direct sum (1, 1, 1) $+$ (1, 3, 1),
        \item[(2, 5, 1)]:
            possible -- direct sum (1, 1) $+$ (1, 4, 1),
        \item[(3, 2, 3)]:
            impossible:
            Let us assume that there exists a self-dual module with the local Hilbert function (3, 2, 3).
            Then by Lemma~\ref{lemma-symmetric-function}, Iarrobino's symmetric decomposition must be trivial (i.e., has only one row).
            Thus, by Lemma~\ref{lemma-graded-self-dual} there exists a self-dual graded module with the local Hilbert function (3, 2, 3).
            We have $h(1) = 2$, so the space of the first-order derivatives has a degree $2$.
            We have $h(2) = 3$, so we have three polynomials of degree $3$.
            If all these polynomials have one-dimensional subspaces of derivatives, then they are linearly dependent modulo constants -- a contradiction.
            So there exists a polynomial of degree $2$ that has a two-dimensional subspace of derivatives.
            From $h(1) = 2$, the subspace of the polynomials of degree $1$ is two-dimensional, so it is generated by at most two variables.
            Since the module is graded, it follows that we only have two variables in all degrees.
            Let us take the polynomial that has a two-dimensional subspace of derivatives.
            It has a form $v_1 x_1^2 + v_2 x_2^2 + v_3 x_1x_2$, where $v_i \in \mathbb{R}^3$.
            Now let us take the subspace of the polynomials of degree $2$ -- they have a form $w_1 x_1^2 + w_2 x_2^2 + w_3 x_1x_2$.
            Their derivatives are in the space $(v_1 x_1 + v_3 x_2; v_2 x_2 + v_3 x_1)$.
            From derivation with respect to $x_1$ we obtain $w_1 = k_1v_1 + l_1v_3, w_3 = k_1v_3 + l_1v_2$ and from derivation with respect to $x_2$ we obtain $w_2 = k_2v_3 + l_2 v_2, w_3 = k_2v_1 + l_2v_3$, where $k_1, k_2, l_1, l_2$ are constants.
            So we obtain the equation $k_2v_1 + l_2v_3 = k_1v_3 + l_1v_2$.
            Without losing generality, let us assume $v_1 \neq 0$.
            Let us observe that the subspace of degree $2$ has dimension $3$, so it is spanned by $k_1, l_1, l_2$.
            We have $k_2 v_{1i} = k_1 v_{3i} + l_1 v_{2i} - l_2v_{3i}$.
            Taking $l_1 = 0, l_2 = 0, k_1 = 1$ and $k_1 = 0, l_2 = 0, l_1 = 1$ we obtain $v_{3i} = b v_{1i}, v_{2i} = a v_{1i}$, so $v_2 = av_1, v_3 = bv_1$.
            Thus, we have that our polynomial that has a two-dimensional subspace of derivatives is of the form $v_1 x_1^2 + a v_1 x_2^2 + b v_1 x_1x_2$.
            The subspace of the polynomials of degree $1$ is then $(v_1 x_1 + b v_1 x_2; a v_1 x_2 + b v_1 x_1)$ and the subspace of degree $0$ is $(v_1)$. From self-duality, there are exactly three generators, and we already have three of them (the polynomials of degree $2$), so there is only one polynomial of degree $0$, which contradicts $h(0) = 3$.
        \item[(3, 4, 1)]:
            possible -- direct sum (1, 1) $+$ (2, 3, 1).
    \end{enumerate}

\section{Future work}
\label{section-future-work}

It would be interesting to understand the case of a prime characteristic of the field $\mathbb{K}$. This requires slight changes in the apolarity notion.
S. Kleiman and J.O. Kleppe generalized the Macaulay duality over more general rings \cite{kleiman2025macaulay}. The generalization of our work in that direction will be very interesting and seems possible; we leave it for future work as well.

\printbibliography
\end{document}